\documentclass{amsart}
\usepackage{euscript,amsmath, amssymb, amsfonts}

\numberwithin{equation}{section}

\newtheorem{theorem}{Theorem}[section]
\newtheorem{lemma}[theorem]{Lemma}
\newtheorem{proposition}[theorem]{Proposition}
\newtheorem{corollary}[theorem]{Corollary}

\theoremstyle{definition}
\newtheorem{definition}[theorem]{Definition}
\newtheorem{example}[theorem]{Example}

\newcommand{\R}{{\mathbb R}}
\newcommand{\N}{{\mathbb N}}
\newcommand{\mb}[1]{\mathbf{#1}}
\newcommand{\fm}[2]{#1\!:\!#2}

\begin{document}

\title{Completion procedures in measure theory}

\author{A.G.~Smirnov}
\address{I.~E.~Tamm Theory Department, P.~N.~Lebedev
	Physical Institute, Leninsky prospect 53, Moscow 119991, Russia}
\email{smirnov@lpi.ru}	

\author{M.S.~Smirnov}
\address{Marchuk Institute of Numerical Mathematics RAS, Gubkin str. 8, Moscow 119333, Russia}
\email{matsmir98@gmail.com}	


\begin{abstract}
	We propose a unified treatment of extensions of group-valued contents (i.e., additive set functions defined on a ring) by means of adding new null sets. Our approach is based on the notion of a completion ring for a content $\mu$. With every such ring $\mathcal N$, an extension of $\mu$ is naturally associated which is called the $\mathcal N$-completion of $\mu$. The $\mathcal N$-completion operation comprises most previously known completion-type procedures and also gives rise to some new extensions, which may be useful for constructing counterexamples in measure theory. We find a condition ensuring that $\sigma$-additivity of a content is preserved under the $\mathcal N$-completion and establish a criterion for the $\mathcal N$-completion of a measure to be again a measure. 
\end{abstract}

\maketitle
	
\section{Introduction}
	
	Let $\mathfrak A$ be an Abelian group. We say that a map $\mu$ is an $\mathfrak A$-valued content if it is an $\mathfrak A$-valued additive set function and its domain of definition $D_\mu$ is a ring of sets. If $\mathfrak A$ is a Hausdorff topological Abelian group, then a $\sigma$-additive $\mathfrak A$-valued content is called an $\mathfrak A$-valued $\sigma$-content. An $\mathfrak A$-valued $\sigma$-content $\mu$ is said to be an $\mathfrak A$-valued measure if $D_\mu$ is a $\delta$-ring.  
	
	Given a content $\mu$, a set $N\in D_\mu$ is called a $\mu$-null set if $\mu(N') = 0$ for every $N'\in D_\mu$ such that $N'\subset N$. The set of all $\mu$-null sets is denoted by $\mathcal N_\mu$. We define
	\begin{align}
	& \hat{\mathcal N}_\mu = \{N: N\subset N'\mbox{ for some } N'\in \mathcal N_\mu\},\nonumber \\
	&\mathcal N_0[\mu] = \{N\subset S[\mu]: N\cap B\in \mathcal N_\mu\mbox{ for every} B\in D_\mu\},\label{N0mu}\\
	&\mathcal N[\mu] = \{N\subset S[\mu]: N\cap B\in \hat{\mathcal N}_\mu\mbox{ for every} B\in D_\mu\},\label{Nmu}
	\end{align}
	where $S[\mu] = \bigcup D_\mu$. We call the elements of $\hat{\mathcal N}_\mu$, $\mathcal N_0[\mu]$, and $\mathcal N[\mu]$ $\mu$-negligible, locally $\mu$-null, and locally $\mu$-negligible sets respectively. Every content $\mu$ can be naturally extended to contents that have $\hat{\mathcal N}_\mu$, $\mathcal N_0[\mu]$, and $\mathcal N[\mu]$ as their sets of null sets. We denote these extensions by $\mathfrak C\mu$, $\mathfrak S\mu$, and $\overline{\mathfrak C}\mu$ and call them the completion, the saturation, and the essential completion of $\mu$ respectively. 
	
	The sets $\hat{\mathcal N}_\mu$, $\mathcal N_0[\mu]$, and $\mathcal N[\mu]$, as well as the extensions $\mathfrak C\mu$, $\mathfrak S\mu$, and $\overline{\mathfrak C}\mu$, occur in various contexts in measure theory.\footnote{In many expositions (see, e.g., \cite{Dinculeanu,Halmos}), measure-theoretic constructions involve an arbitrary set $X\supset S[\mu]$. In this case, $S[\mu]$ should be replaced with $X$ in~(\ref{N0mu}) and~(\ref{Nmu}). In this paper, we confine ourselves to the case $X = S[\mu]$, which seems to be sufficient for all practical purposes.} The completion $\mathfrak C\mu$ of a positive measure is a standard measure-theoretic construction (see, e.g., Halmos~\cite[Sec.~13]{Halmos}). A discussion of $\mathfrak C\mu$ (under the name of null-completion) for vector-valued contents can be found in~\cite{Butkovic}. The null sets defined by Segal~\cite{Segal} for a positive measure $\mu$ coincide with the elements of $\mathcal N_0[\mu]$ and the extension procedure given in~\cite[Definition~2.2]{Segal} is closely related to $\mathfrak S\mu$. The same definition of null sets is adopted by Masani and Niemi~\cite{MasaniNiemi1,MasaniNiemi2} for vector measures. The completion of a content $\mu$ described by Dinculeanu~\cite[Sec.~I.5.2]{Dinculeanu} is actually equal to $\mathfrak C\mathfrak S\mu$. At the same time, if $\mu$ is a positive measure, then $\mu$-negligible sets defined in~\cite[Sec.~I.5.4]{Dinculeanu} are precisely the elements of $\mathcal N[\mu]$ and thus correspond to $\overline{\mathfrak C}\mu$. Similarly, the treatments of scalar Radon measures by Bourbaki~\cite{BourbakiIntegration} and vector measures in general topological vector spaces by Thomas~\cite{Thomas2012} involve elements of $\mathcal N[\mu]$, where they are called locally negligible sets and nul sets respectively. In this connection, the question naturally arises whether Dinculeanu's completion $\mathfrak C\mathfrak S\mu$ is always equal to the essential completion~$\overline{\mathfrak C}\mu$.     
	
	In this paper, we propose a unified approach to various completion procedures based on the notion of a \textit{completion ring}. Given an $\mathfrak A$-valued content $\mu$, every completion ring $\mathcal N$ for $\mu$ gives rise to a naturally defined $\mathfrak A$-valued content that extends $\mu$ and has $\mathcal N$ as its set of null sets. We call this content the $\mathcal N$-completion of $\mu$ and denote it by $\mathfrak C_{\mathcal N}\mu$. The sets $\hat{\mathcal N}_\mu$, $\mathcal N_0[\mu]$, and $\mathcal N[\mu]$ turn out to be completion rings for $\mu$ and $\mathfrak C\mu$, $\mathfrak S\mu$, and $\overline{\mathfrak C}\mu$ are precisely the corresponding $\mathcal N$-completions of $\mu$. Apart from these ``natural'' completion rings, there are many other ones and, in general, $\mathcal N$-completions may have rather exotic properties. In particular, the well-known counterexample to the Radon--Nikodym theorem given by Halmos~\cite[p.~131]{Halmos} can be conveniently expressed in terms of completions with respect to suitable completion rings. We use an $\mathcal N$-completion of this kind to produce a positive measure $\mu$ such that $\mathfrak C\mathfrak S\mu\neq \overline{\mathfrak C}\mu$, thus answering the above question in the negative.
	
	Elementary examples show that, in general, the $\mathcal N$-completion of a measure can be neither $\sigma$-additive nor defined on a $\delta$-ring. We find simple conditions ensuring that the $\mathcal N$-completions of $\sigma$-contents and measures are again $\sigma$-contents and measures (the condition for measures is actually necessary and sufficient). In particular, if $\mu$ is a $\sigma$-content (a measure), then each of the natural completions $\mathfrak C\mu$, $\mathfrak S\mu$, and $\overline{\mathfrak{C}}\mu$ is also a $\sigma$-content (resp., a measure).
	
	It is worth noting that the $\mathcal N$-completions can be defined without appealing to the additive structure of $\mathfrak A$. This requires a slight modification of the definition of null sets (as defined above, they depend on $\mathfrak A$ via its neutral element). Such a modification allows us to introduce the notion of \textit{a map admitting completion} that does not depend on the group $\mathfrak A$. It turns out that all $\mathfrak A$-valued contents admit completion and the basic properties of $\mathcal N$-completions can be derived for maps admitting completion in place of contents.
	
	It is important that most integration theories actually involve some completion procedure. Indeed, to every integral, there corresponds an extension of the initial measure to the $\delta$-ring consisting of sets whose characteristic functions are integrable. This extension is usually given by one of the natural completion procedures. In particular, the integrals constructed in~\cite{MasaniNiemi1,MasaniNiemi2} and~\cite{Dinculeanu,Thomas2012} involve saturation and essential completion respectively. In a forthcoming paper, we shall employ both these procedures to generalize Dinculeanu's bilinear integral~\cite{Dinculeanu} for the case of arbitrary locally convex spaces. 
	
	The paper is organized as follows. In Section~\ref{s2}, we fix our notation and give basic definitions related to rings of sets, contents, and measures. In Section~\ref{s3}, we define completion rings and $\mathcal N$-completions and prove that $\mathcal N$ is precisely the set of all null sets of $\mathfrak C_{\mathcal N}\mu$. We also establish the following transitivity property: if $\mathcal N$ and $\mathcal N'$ are completion rings for $\mu$ and $\mathcal N\subset \mathcal N'$, then $\mathfrak C_{\mathcal N'}\mathfrak C_{\mathcal N}\mu = \mathfrak C_{\mathcal N'}\mu$. In Section~\ref{s4}, we consider the natural completions $\mathfrak C\mu$, $\mathfrak S\mu$, and $\overline{\mathfrak C}\mu$ and find all identities between compositions of these operations. All results of Sections~\ref{s3} and~\ref{s4} are formulated in terms of maps admitting completion and do not depend on the group $\mathfrak A$. In Section~\ref{s5}, we show that $\mathcal N$-completions of contents are contents and give a sufficient condition for an $\mathcal N$-completion of a $\sigma$-content to be a $\sigma$-content. We also discuss the relation between $\mathcal N$-completions and topologies induced by contents on their domains. In Section~\ref{s6}, we establish a criterion for an $\mathcal N$-completion of a measure to be a measure. In Section~\ref{s7}, we give some examples of $\mathcal N$-completions. In particular, we construct a positive measure such that $\mathfrak C\mathfrak S\mu\neq \overline{\mathfrak C}\mu$ (Example~\ref{e1}) and rederive the Halmos counterexample to the Radon--Nikodym theorem in terms of suitable $\mathcal N$-completions (Example~\ref{e2}). We also provide some evidence showing that our general results apparently cannot be considerably strengthened (Examples~\ref{e3}, \ref{e4}, and~\ref{e5}).
	
\section{Preliminaries}
\label{s2}
	
	We let $\R$ and $\N$ denote the set of real numbers and the set of strictly positive integer numbers respectively. A set $A$ is called countable if there is a bijection between $A$ and a subset of $\N$ (thus, finite sets are countable). Given a map $f$, we let $D_f$ and $R_f$ denote the domain of definition and the range of $f$ respectively. We write $\fm{f}{I}$ if $f$ is a map and $D_f = I$. The words ``map'' and ``family'' are used interchangeably. A family $\fm{f}{I}$ is called disjoint if $f(i)\cap f(j) = \varnothing$ for every $i,j\in I$ such that $i\neq j$. We say that a family $\fm{f}{I}$ is a partition of a set $A$ if it is disjoint and $A = \bigcup_{i\in I}f(i)$. A family $\fm{f}{I}$ is called finite (countable) if $I$ is finite (resp., countable).
	
	A nonempty set $\mathcal Q$ is called a ring if $A\cup B\in\mathcal Q$ and $A\setminus B\in\mathcal Q$ for every $A,B\in\mathcal Q$.
	A ring $\mathcal Q$ is called a $\delta$-ring ($\sigma$-ring) if the intersection (resp., union) of every nonempty countable family of elements of $\mathcal Q$ belongs to $\mathcal Q$. A ring ($\sigma$-ring) that has the greatest element with respect to inclusion is called an algebra (resp., $\sigma$-algebra). Every $\sigma$-ring is a $\delta$-ring. For every set $\mathcal K$, there is a unique ring ($\sigma$-ring, $\delta$-ring) $\mathcal Q\supset \mathcal K$ such that $\mathcal Q\subset \mathcal Q'$ for every ring (resp., $\sigma$-ring, $\delta$-ring) $\mathcal Q'\supset\mathcal K$. It is called the ring (resp., $\sigma$-ring, $\delta$-ring) generated by the set $\mathcal K$ and will be denoted by $\kappa(\mathcal K)$ (resp., $\sigma(\mathcal K)$, $\delta(\mathcal K)$).

	\begin{lemma}\label{l_NcapB}
		Let $\mathcal Q$ and $\mathcal N$ be sets such that $N\cap B\in\mathcal N$ for every $B\in\mathcal Q$ and $N\in\mathcal N$. If $B\in\sigma(\mathcal Q)$ and $N\in\sigma(\mathcal N)$, then $N\cap B\in\sigma(\mathcal N)$. If $B\in \sigma(\mathcal Q)$ and $B\subset N$ for some $N\in\sigma(\mathcal N)$, then $B\in \sigma(\mathcal N)$. 
	\end{lemma}
	\begin{proof}
		Let $B\in\mathcal Q$ and $\mathcal K = \{N\in\sigma(\mathcal N): N\cap B\in \sigma(\mathcal N)\}$. Since $\mathcal N\subset \mathcal K$ and $\mathcal K$ is a $\sigma$-ring, we have $\mathcal K = \sigma(\mathcal N)$. This proves that $N\cap B\in\sigma(\mathcal N)$ for every  $B\in\mathcal Q$ and $N\in\sigma(\mathcal N)$. Let $N\in\sigma(\mathcal N)$ and $\mathcal R = \{B\in\sigma(\mathcal Q): N\cap B\in\sigma(\mathcal N)\}$. Since $\mathcal Q\subset \mathcal R$ and $\mathcal R$ is a $\sigma$-ring, we have $\mathcal R = \sigma(\mathcal Q)$. Thus, $N\cap B\in\sigma(\mathcal N)$ for every $B\in\sigma(\mathcal Q)$ and $N\in\sigma(\mathcal N)$.  
		If $B\in \sigma(\mathcal Q)$, $N\in\sigma(\mathcal N)$, and $B\subset N$, then $B\in\sigma(\mathcal N)$ because $B = N\cap B$. 
	\end{proof}
	
	The proof of the next elementary statement can be found in Sec.~I.1.3 of~\cite{Dinculeanu}.
	
	\begin{lemma}\label{ll1}
		Let $\mathcal Q$ be a $\delta$-ring. Then the following statements hold:
		\begin{itemize}
			\item [(i)] Let $B\in\mathcal Q$ and $\mb A\colon I\to \mathcal Q$ be a countable family such that $\mb A(i)\subset B$ for all $i\in I$. Then $\bigcup_{i\in I} \mb A(i) \in \mathcal Q$.
			\item [(ii)] $\sigma(\mathcal Q)$ is the set of all countable unions of elements of $\mathcal Q$.
			\item[(iii)] If $A\in \sigma(\mathcal Q)$ and $A\subset B$ for some $B\in \mathcal Q$, then $A\in \mathcal Q$.
		\end{itemize}
	\end{lemma}
	
	\begin{definition}\label{dd4}
		Let $\mathfrak A$ be an Abelian group. A map $\mu$ is said to be an $\mathfrak A$-valued additive function if $R_\mu\subset \mathfrak A$ and 
		$\mu(A) = \sum_{i\in I}\mu(\mb A(i))$ for every $A\in D_\mu$ and every finite partition $\mb A\colon I\to D_\mu$ of $A$. A map $\mu$ is called an $\mathfrak A$-valued content if $\mu$ is an $\mathfrak A$-valued additive function and $D_\mu$ is a ring.
	\end{definition}
	
	Let $\mathfrak A$ be a Hausdorff topological Abelian group (HTAG for short). A family $f\colon I\to \mathfrak A$ is called summable in $\mathfrak A$ if the net $J\to \sum_{i\in J}f(i)$, where $J$ ranges over finite subsets of $I$, is convergent in $\mathfrak A$. In this case, its limit is denoted by $\sum_{i\in I}f(i)$. 
	
	\begin{definition}\label{dd5}
		Let $\mathfrak A$ be a HTAG. A map $\mu$ is said to be an $\mathfrak A$-valued $\sigma$-additive function if $R_\mu\subset \mathfrak A$ and the following condition holds:
		\begin{itemize}
			\item [$(\sigma)$] If $A\in D_\mu$ and $\mb A\colon I\to D_\mu$ is a countable partition of $A$, then the family $\mu\circ\mb A$ is summable in $\mathfrak A$ and $\mu(A) = \sum_{i\in I}\mu(\mb A(i))$.
		\end{itemize}
		A map $\mu$ is called an $\mathfrak A$-valued $\sigma$-content (an $\mathfrak A$-valued measure) if $\mu$ is an $\mathfrak A$-valued $\sigma$-additive function and $D_\mu$ is a ring (resp., a $\delta$-ring). 
	\end{definition}
	
	We say that $\mu$ is a positive content ($\sigma$-content, measure) if $\mu$ is an $\R$-valued content (resp., $\sigma$-content, measure) and $\mu(A)\geq 0$ for every $A\in D_\mu$. 
	 
	A set $\mathcal K$ is called hereditary if every subset of every element of $\mathcal K$ belongs to $\mathcal K$. For every set $\mathcal K$, we define the hereditary set $\hat{\mathcal K}$ by the equality
	\[
	\hat{\mathcal K} = \{A : A\subset B\mbox{ for some } B\in\mathcal K\}.
	\]
	Clearly, a set $\mathcal K$ is hereditary if and only if $\mathcal K = \hat{\mathcal K}$.
	
	Given sets $A$ and $B$, we let $A\Delta B$ denote their symmetric difference,
	\[
	A\Delta B = (A\setminus B)\cup(B\setminus A).
	\]
	For every sets $A$, $B$, and $C$, we have $A\setminus C\subset (A\setminus B)\cup (B\setminus C)$ and, therefore,
	\begin{equation}\label{trans}
	A\Delta C\subset (A\Delta B)\cup (B\Delta C).
	\end{equation}
	It is straightforward to verify that the symmetric difference is associative and the intersection is distributive with respect to the symmetric difference:
	\begin{align}
	&(A\Delta B)\Delta C = A\Delta(B\Delta C),\label{assoc}\\
	&(A\Delta B)\cap C = (A\cap C)\Delta(B\cap C)\label{distr}
	\end{align}
	for every sets $A$, $B$, and $C$. If $\fm{\mb A}{I}$ and $\fm{\mb B}{I}$ are families of sets, then
	\begin{align}
	&\left(\bigcup\nolimits_{i\in I}\mb A(i)\right)\Delta \left(\bigcup\nolimits_{i\in I}\mb B(i)\right) \subset \bigcup\nolimits_{i\in I}\mb A(i)\Delta\mb B(i),\label{0}\\
	&\left(\bigcap\nolimits_{i\in I}\mb A(i)\right)\Delta \left(\bigcap\nolimits_{i\in I}\mb B(i)\right) \subset \bigcup\nolimits_{i\in I}\mb A(i)\Delta\mb B(i).\label{01}
	\end{align}

	For every sets $\mathcal Q$ and $\mathcal N$, we define
	\[
	\Delta(\mathcal Q,\mathcal N) = \{A: A\Delta B\in\mathcal N\mbox{ for some }B\in \mathcal Q\}.
	\]
	By~(\ref{assoc}), we have $A\Delta B = N$ if and only if $A = B\Delta N$ and, hence, 	
	\begin{equation}\label{DeltaQN}
	\Delta(\mathcal Q,\mathcal N) = \{A: A = B\Delta N \mbox{ for some }B\in\mathcal Q\mbox{ and }N\in \mathcal N\}. 
	\end{equation}
	Clearly, $\Delta(\mathcal Q,\mathcal N) = \Delta(\mathcal N,\mathcal Q)$ for every $\mathcal Q$ and $\mathcal N$.
	
	We note that a nonempty set $\mathcal Q$ is a ring if and only if the sets $A_1\Delta A_2$ and $A_1\cap A_2$ belong to $\mathcal Q$ for every $A_1,A_2\in\mathcal Q$.
	
	\begin{lemma}\label{l_DeltaQN}
		Let $\mathcal Q$ and $\mathcal N$ be rings such that $N\cap B\in\mathcal N$ for every $B\in\mathcal Q$ and $N\in \mathcal N$ and let $\mathcal K = \Delta(\mathcal Q,\mathcal N)$. Then $\mathcal K = \kappa(\mathcal Q\cup\mathcal N)$ and, in particular, $\mathcal K$ is a ring. If $A\in \mathcal K$ and $N\in \mathcal N$, then $N\cap A\in\mathcal N$. If $A\in \mathcal K$ and $A\subset N$ for some $N\in\mathcal N$, then $A\in\mathcal N$.
	\end{lemma}
	\begin{proof}
		Since $\varnothing\in\mathcal Q\cap\mathcal N$, it follows from~(\ref{DeltaQN}) that $\mathcal Q\cup\mathcal N\subset\mathcal K$. Let $A_1,A_2\in\mathcal K$. By~(\ref{DeltaQN}), there are $B_1,B_2\in\mathcal Q$ and $N_1,N_2\in\mathcal N$ such that $A_1 = B_1 \Delta N_1$ and $A_2 = B_2 \Delta N_2$. By~(\ref{assoc}) and~(\ref{distr}), we obtain $A_1\Delta A_2 = (B_1\Delta B_2)\Delta(N_1\Delta N_2)$ and $A_1\cap A_2 = (B_1\cap B_2)\Delta N$, where $N = (B_1\cap N_2)\Delta (B_2\cap N_1)\Delta (N_1\cap N_2)$. Since $B_1\cap N_2$ and $B_2\cap N_1$ belong to $\mathcal N$ and $\mathcal Q$ and $\mathcal N$ are rings, the sets $B_1\Delta B_2$ and $B_1\cap B_2$ belong to $\mathcal Q$ and the sets $N_1\Delta N_2$ and $N$ belong to $\mathcal N$. In view of~(\ref{DeltaQN}), we conclude that $A_1\Delta A_2, A_1\cap A_2\in\mathcal K$, i.e., $\mathcal K$ is a ring. As $\mathcal K \subset \kappa(\mathcal Q\cup\mathcal N)$ by~(\ref{DeltaQN}), we obtain $\mathcal K = \kappa(\mathcal Q\cup\mathcal N)$. Let $A\in \mathcal K$ and $N\in\mathcal N$ and let $B\in\mathcal Q$ and $N'\in\mathcal N$ be such that $A = B\Delta N'$. Then $N\cap A = (N\cap B)\Delta (N\cap N')$ by~(\ref{distr}). Since $N\cap B\in\mathcal N$ and $\mathcal N$ is a ring, it follows that $N\cap A\in\mathcal N$. If $A\in \mathcal K$, $N\in\mathcal N$, and $A\subset N$, then $A\in\mathcal N$ because $A = N\cap A$.
	\end{proof}

\section{Completion rings}
\label{s3}
	
	\begin{definition}
		Let $\mu$ be a map such that $\varnothing\in D_\mu$. We say that $N\in D_\mu$ is a $\mu$-null set if $\mu(N')=\mu(\varnothing)$ for every $N'\in D_\mu$ such that $N'\subset N$. The set of all $\mu$-null sets is denoted by $\mathcal N_\mu$. 
	\end{definition}
	
	If $\mathfrak A$ is an Abelian group and $\mu$ is an $\mathfrak A$-valued additive function such that $\varnothing \in D_\mu$, then $\mu(\varnothing) = \mu(\varnothing) + \mu(\varnothing)$ and, therefore, $\mu(\varnothing)=0$. This implies that
	\[
	\mathcal N_\mu = \{N\in D_\mu : \mu(N') = 0\mbox{ for every }N'\in D_\mu \mbox{ such that } N'\subset N\},
	\]
	i.e., we recover the definition given in the introduction.
	
	\begin{definition}\label{dd7}
		A map $\mu$ is said to admit completion if $D_\mu$ is a ring and $\mu(B_1)=\mu(B_2)$ for every $B_1,B_2\in D_\mu$ such that $B_1\Delta B_2\in\mathcal N_\mu$. If, in addition, $\mathcal N_\mu$ is hereditary, then $\mu$ is called complete. 
	\end{definition}
	
	\begin{proposition}\label{p_Nring}
		Let $\mu$ be a map admitting completion. Then $\mathcal N_\mu$ is a ring. If $D_\mu$ is a $\delta$-ring, then so is $\mathcal N_\mu$. 
	\end{proposition}
	\begin{proof}
		Let $N_1,N_2\in\mathcal N_\mu$, $N = N_1\cup N_2$, and $\tilde N = N_1\setminus N_2$. Since $D_\mu$ is a ring, we have $N,\tilde N\in D_\mu$. As $\tilde N\subset N_1$, it follows that $\tilde N\in\mathcal N_\mu$. Let $N'\in D_\mu$ be such that $N'\subset N$ and let $A = N_1\cap N'$. Then $A\in D_\mu$ and $A\subset N_1$ and, hence, $A\in \mathcal N_\mu$. Let $B = A\Delta N' = N'\setminus N_1$. Then $B\in D_\mu$ and $B\subset N_2$, whence $B\in\mathcal N_\mu$. This implies that $\mu(N') = \mu(A) = \mu(\varnothing)$ and, therefore, $N\in \mathcal N_\mu$. Thus, $\mathcal N_\mu$ is a ring. Let $D_\mu$ be a $\delta$-ring and $\mb N\colon I\to \mathcal N_\mu$ be a nonempty countable family. Since $\mb N(i)\in D_\mu$ for all $i\in I$, the set $N = \bigcap_{i\in I}\mb N(i)$ belongs to $D_\mu$. As $N\subset \mb N(i)$ for every $i\in I$, this implies that $N\in\mathcal N_\mu$. This means that $\mathcal N_\mu$ is a $\delta$-ring.   
	\end{proof}
	
	\begin{definition}\label{d_cring}
		Let $\mu$ be a map admitting completion and $\mathcal N$ be a ring such that $N\cap B\in\mathcal N$ for every $N\in\mathcal N$ and $B\in D_\mu$. We say that $\mathcal N$ is a weak completion ring for $\mu$ if $\mathcal N\cap D_\mu\subset \mathcal N_\mu$. We say that $\mathcal N$ is a completion ring for $\mu$ if $\mathcal N\cap D_\mu = \mathcal N_\mu$.  
	\end{definition}

	By Proposition~\ref{p_Nring}, $\mathcal N_\mu$ is a completion ring for every map $\mu$ admitting completion.
	
	\begin{lemma}\label{l_compl}
		Let $\mathcal N$ be a weak completion ring for $\mu$. 
		\begin{enumerate}
			\item[(i)] Let $B\in D_\mu$ and $B\subset N$ for some $N\in\mathcal N$. Then $B\in\mathcal N_\mu$.
			\item[(ii)] There is a unique map $\nu$ such that $D_\nu = \Delta(D_\mu,\mathcal N)$ and $\nu(A) = \mu(B)$ whenever $A\in D_\nu$, $B\in D_\mu$, and $A\Delta B\in\mathcal N$.
		\end{enumerate}
	\end{lemma}
	\begin{proof}
		(i) Since $B= N\cap B$, we have $B\in\mathcal N$, whence $B\in\mathcal N_\mu$ because $\mathcal N\cap D_\mu\subset \mathcal N_\mu$.
		\par\medskip\noindent
		(ii) It suffices to prove the existence of $\nu$, its uniqueness being obvious. Let $\mathcal Q = \Delta(D_\mu,\mathcal N)$ and $\tau\colon \mathcal Q\to D_\mu$ be such that $A\Delta\tau(A)\in\mathcal N$ for every $A\in \mathcal Q$. We set $\nu = \mu\circ\tau$. We obviously have $D_\nu = \mathcal Q$. Let $A\in D_\nu$, $B\in D_\mu$, and $A\Delta B\in\mathcal N$. Let $B' = \tau(A)$ and $N = (A\Delta B)\cup (A\Delta B')$. Since $A\Delta B'\in\mathcal N$ and $\mathcal N$ is a ring, we have $N\in\mathcal N$. As $B\Delta B'\subset N$ by~(\ref{trans}) and $B\Delta B'\in D_\mu$, it follows from~(i) that $B\Delta B'\in \mathcal N_\mu$. Hence, $\nu(A) = \mu(B') = \mu(B)$ because $\mu$ admits completion.  
	\end{proof}
	
	\begin{definition}\label{d_Ncompl}
		Let $\mathcal N$ be a weak completion ring for $\mu$. The map $\nu$ satisfying the conditions of Lemma~\ref{l_compl}(ii) is called the $\mathcal N$-completion of $\mu$ and is denoted by~$\mathfrak C_{\mathcal N}\mu$.
	\end{definition}
	
	\begin{proposition}\label{p_weak}
		Let $\mathcal N$ be a weak completion ring for $\mu$. Then $\mathcal N' = \Delta(\mathcal N,\mathcal N_\mu)$ is a completion ring for $\mu$ and $\mathfrak C_{\mathcal N}\mu = \mathfrak C_{\mathcal N'}\mu$.
	\end{proposition}
	\begin{proof}
		Since $\mathcal N_\mu$ is a ring by Proposition~\ref{p_Nring} and $\mathcal N' = \Delta(\mathcal N_\mu,\mathcal N)$, Lemma~\ref{l_DeltaQN} implies that $\mathcal N'$ is a ring.
		Let $N\in \mathcal N'$ and $B\in D_\mu$. Then $N = N_1\Delta N_2$, where $N_1\in\mathcal N$ and $N_2\in\mathcal N_\mu$. By~(\ref{distr}), we have $N\cap B = (N_1\cap B)\Delta (N_2\cap B)$ and, hence, $N\cap B\in \mathcal N'$ because $N_1\cap B\in \mathcal N$ and $N_2\cap B\in\mathcal N_\mu$. We obviously have $\mathcal N_\mu\subset \mathcal N'\cap D_\mu$. Let $N\in \mathcal N'\cap D_\mu$ and let $N_1$ and $N_2$ be as above. By~(\ref{assoc}), we have $N_1 = N\Delta N_2$ and, hence, $N_1\in D_\mu$. This implies that $N_1\in \mathcal N_\mu$ because $\mathcal N$ is a weak completion ring for $\mu$ and, therefore, $N\in\mathcal N_\mu$. Thus, $\mathcal N_\mu = \mathcal N'\cap D_\mu$, i.e., $\mathcal N'$ is a completion ring for $\mu$. Let $\nu = \mathfrak C_{\mathcal N}\mu$ and $\nu' = \mathfrak C_{\mathcal N'}\mu$. Let $A\in D_\nu$ and let $B\in D_\mu$ be such that $A\Delta B\in \mathcal N$. Since $\mathcal N\subset \mathcal N'$, we have $A\Delta B\in\mathcal N'$ and, hence, $A\in D_{\nu'}$ and $\nu'(A) = \mu(B) = \nu(A)$. This means that $\nu'$ is an extension of $\nu$ and it remains to show that $D_{\nu'}\subset D_\nu$. Let $A\in D_{\nu'}$ and let $B\in D_\mu$ and $N\in \mathcal N'$ be such that $A = B\Delta N$. Let $N_1\in\mathcal N$ and $N_2\in\mathcal N_\mu$ be such that $N = N_1\Delta N_2$. By~(\ref{assoc}), we obtain $A = B'\Delta N_1$, where $B' = B\Delta N_2$. Since $B'\in D_\mu$, we conclude that $A\in D_\nu$.  
	\end{proof}
	
	In view of Proposition~\ref{p_weak}, using weak completion rings gives rise to no new objects in comparison with completion rings. For this reason, we shall confine ourselves to completion rings in the rest of the paper. By means of Proposition~\ref{p_weak}, the results given below can be restated in terms of weak completion rings (in general, in a more cumbersome form). 

	\begin{lemma}\label{l_cring}
		Let $\mathcal N$ be a completion ring for $\mu$ and $\nu = \mathfrak C_{\mathcal N}\mu$. If $A\in D_\nu$ and $N\in\mathcal N$, then $N\cap A\in\mathcal N$. If $A\in D_\nu$ and $A\subset N$ for some $N\in\mathcal N$, then $A\in\mathcal N$. 
	\end{lemma}
	\begin{proof}
		Since $D_\nu = \Delta(D_\mu,\mathcal N)$, the statement follows from Lemma~\ref{l_DeltaQN}.
	\end{proof}	

	Given a map $\mu$, we set $S[\mu] = \bigcup D_\mu$. 
	
	\begin{theorem}\label{p_Ncompl}
		Let $\mathcal N$ be a completion ring for $\mu$ and $\nu = \mathfrak C_{\mathcal N}\mu$. Then $D_\nu = \kappa(D_\mu\cup\mathcal N)$, $\nu$ admits completion and is an extension of $\mu$, $\mathcal N_\nu = \mathcal N$, and $S[\nu] = S[\mu]\cup\bigcup\mathcal N$. The equality $\nu = \mu$ holds if and only if $\mathcal N = \mathcal N_\mu$.
	\end{theorem}
	\begin{proof}
		Since $D_\nu = \Delta(D_\mu,\mathcal N)$, it follows from Lemma~\ref{l_DeltaQN} that $D_\nu = \kappa(D_\mu\cup\mathcal N)$ and, in particular, that $D_\nu$ is a ring and $D_\mu\cup\mathcal N\subset D_\nu$. For every $B\in D_\mu$, we have $\nu(B) = \mu(B)$ because $B\Delta B =\varnothing\in \mathcal N$. This means that $\nu$ is an extension of $\mu$.
		
		Let $N\in\mathcal N$ and let $N'\in D_\nu$ be such that $N'\subset N$. Then $N'\in\mathcal N$ by Lemma~\ref{l_cring}. Since $N'\Delta\varnothing = N'\in\mathcal N$, we obtain $\nu(N') = \mu(\varnothing) = \nu(\varnothing)$. As $N\in D_\nu$, it follows that $N\in \mathcal N_\nu$, i.e., $\mathcal N\subset \mathcal N_\nu$. Let $N\in \mathcal N_\nu$. Since $N\in D_\nu$, there is $B\in D_\mu$ such that $N\Delta B\in \mathcal N$. Let $B'\in D_\mu$ and $B'\subset B$. Setting $N' = B'\cap N$, we obtain $N'\Delta B' = B'\setminus N\subset N\Delta B$. The sets $N'$ and $N'\Delta B'$ belong to $D_\nu$ because $D_\nu$ is a ring. Lemma~\ref{l_cring} therefore implies that $N'\Delta B'\in\mathcal N$. Since $N'\subset N$, we have $\mu(B') = \nu(N') = \nu(\varnothing) = \mu(\varnothing)$ and, hence, $B\in\mathcal N_\mu$. As $\mathcal N_\mu\subset \mathcal N$, it follows that $B\in\mathcal N$. This implies that $N\in\mathcal N$ because $N = (N\Delta B)\Delta B$ by~(\ref{assoc}) and $\mathcal N$ is a ring. Thus, $\mathcal N_\nu\subset \mathcal N$ and, therefore, $\mathcal N_\nu = \mathcal N$.    
		
		Let $A_1,A_2\in D_\nu$ and $A_1\Delta A_2\in \mathcal N_\nu$. Then $A_1\Delta A_2\in \mathcal N$. Let $B_1,B_2\in D_\mu$ be such that the sets $A_1\Delta B_1$ and $A_2\Delta B_2$ belong to $\mathcal N$. By~(\ref{trans}), we have $B_1\Delta B_2\subset (B_1\Delta A_1)\cup (A_1\Delta A_2)\cup (A_2\Delta B_2)$. The set in the right-hand side of this inclusion belongs to $\mathcal N$ because $\mathcal N$ is a ring. As $B_1\Delta B_2\in D_\mu$, it follows from Lemma~\ref{l_compl}(i) that $B_1\Delta B_2\in\mathcal N_\mu$. Since $\mu$ admits completion, we obtain $\nu(A_1) = \mu(B_1) = \mu(B_2) = \nu(A_2)$. This means that $\nu$ admits completion.
		
		Let $S = S[\mu]\cup\bigcup\mathcal N$. Since $D_\mu\cup\mathcal N\subset D_\nu$, we have $S\subset S[\nu]$. Let $A\in D_\nu$ and $B\in D_\mu$ be such that $A\Delta B\in \mathcal N$. Then $A\subset B\cup(A\Delta B)\subset S$ and, hence, $S[\nu]\subset S$. Thus, $S[\nu] = S$. 
		
		Let $\mathcal N = \mathcal N_\mu$ and $A\in D_\nu$. Then $A\Delta B\in \mathcal N_\mu$ for some $B\in D_\mu$. Since $A = (A\Delta B)\Delta B$ by~(\ref{assoc}), we conclude that $A\in D_\mu$. Thus, $D_\nu\subset D_\mu$ and, hence, $\nu = \mu$. If $\nu = \mu$, then $D_\nu = D_\mu$, whence $\mathcal N\subset D_\mu$. It follows that $\mathcal N = \mathcal N\cap D_\mu = \mathcal N_\mu$.
	\end{proof}
	
	\begin{proposition}\label{p_trans}
		Let $\mathcal N$ be a completion ring for $\mu$ and $\nu = \mathfrak C_{\mathcal N}\mu$.
		\begin{enumerate}
			\item[(i)] $\mathcal N'$ is a completion ring for $\nu$ if and only if $\mathcal N'$ is a completion ring for $\mu$ and $\mathcal N\subset \mathcal N'$.
			\item[(ii)] Let $\mathcal N'$ be a completion ring for $\nu$. Then $\mathfrak C_{\mathcal N'}\nu = \mathfrak C_{\mathcal N'}\mu$. 
		\end{enumerate}
	\end{proposition}
	\begin{proof}
		(i) Let $\mathcal N'$ be a completion ring for $\nu$. Then $\mathcal N_\nu = \mathcal N'\cap D_\nu$. Since 
		\begin{equation}\label{N=N}
		\mathcal N_\nu = \mathcal N
		\end{equation}
		by Theorem~\ref{p_Ncompl}, it follows that $\mathcal N\subset \mathcal N'$. If $B\in D_\mu$ and $N\in\mathcal N'$, then $N\cap B\in\mathcal N'$ because $D_\mu\subset D_\nu$. Since $\mathcal N$ is a completion ring for $\mu$, we have $\mathcal N_\mu = \mathcal N\cap D_\mu$. In view of~(\ref{N=N}), it follows that
		\[
		\mathcal N_\mu = \mathcal N_\nu\cap D_\mu = \mathcal N'\cap D_\nu \cap D_\mu = \mathcal N'\cap D_\mu.
		\] 
		Thus, $\mathcal N'$ is a completion ring for $\mu$.
		
		Conversely, let $\mathcal N'$ be a completion ring for $\mu$ and $\mathcal N\subset \mathcal N'$. Let $A\in D_\nu$ and $N\in\mathcal N'$. Let $B\in D_\mu$ be such that $A\Delta B\in\mathcal N$. As $\mathcal N'$ is a completion ring for $\mu$, we have $N\cap B\in \mathcal N'$. The equality
		\[
		N\cap A = (N\cap (A\Delta B))\Delta (N\cap B),
		\]
		which follows from~(\ref{assoc}) and~(\ref{distr}), therefore implies that $N\cap A\in\mathcal N'$. By~(\ref{N=N}), we have $\mathcal N\subset D_\nu$ and, hence, $\mathcal N\subset \mathcal N'\cap D_\nu$. Let $N\in \mathcal N'\cap D_\nu$. As $N\in D_\nu$, there is $B\in D_\mu$ such that $N\Delta B\in\mathcal N$. By~(\ref{assoc}), we have $B = (B\Delta N)\Delta N$ and, hence, $B\in\mathcal N'$. Since both $\mathcal N$ and $\mathcal N'$ are completion rings for $\mu$, we have $\mathcal N\cap D_\mu = \mathcal N'\cap D_\mu = \mathcal N_\mu$ and, therefore, $B\in\mathcal N$. This implies that $N\in\mathcal N$ because $N = (N\Delta B)\Delta B$ by~(\ref{assoc}). Thus, $\mathcal N'\cap D_\nu\subset\mathcal N$. In view of~(\ref{N=N}), we conclude that $\mathcal N_\nu=\mathcal N'\cap D_\nu$. This means that $\mathcal N'$ is a completion ring for $\nu$.
		\par\medskip\noindent
		(ii) By~(i), $\mathcal N'$ is a completion ring for $\mu$ and $\mathcal N\subset\mathcal N'$. In particular, $\mathfrak C_{\mathcal N'}\mu$ is well defined. Let $\mu' =\mathfrak C_{\mathcal N'}\mu$ and $\nu' = \mathfrak C_{\mathcal N'}\nu$. Let $A\in D_{\mu'}$ and let $B\in D_\mu$ be such that $A\Delta B\in \mathcal N'$. Since $\nu$ is an extension of $\mu$, we have $B\in D_\nu$. This implies that $A\in D_{\nu'}$ and $\nu'(A) = \nu(B) = \mu(B) = \mu'(A)$. Thus, $\nu'$ is an extension of $\mu'$ and it remains to show that $D_{\nu'}\subset D_{\mu'}$. Let $A\in D_{\nu'}$ and let $B\in D_\nu$ be such that $A\Delta B\in\mathcal N'$. Let $C\in D_\mu$ be such that $B\Delta C\in\mathcal N$. By~(\ref{trans}), we have $A\Delta C\subset A\Delta B\cup B\Delta C$. Since the set in the right-hand side of this inclusion belongs to $\mathcal N'$ and $A\Delta C\in D_{\nu'}$, it follows from Lemma~\ref{l_cring} that $A\Delta C\in \mathcal N'$ and, hence, $A\in D_{\mu'}$. Thus, $D_{\nu'}\subset D_{\mu'}$.  
	\end{proof}
	
	\begin{proposition}\label{p_ext}
		Let $\mathcal N$ and  $\mathcal N'$ be completion rings for $\mu$, $\nu = \mathfrak C_{\mathcal N}\mu$, and $\nu' = \mathfrak C_{\mathcal N'}\mu$. Then $\nu'$ is an extension of $\nu$ if and only if $\mathcal N\subset \mathcal N'$. The equality $\nu' = \nu$ holds if and only if $\mathcal N' = \mathcal N$.  
	\end{proposition}
	\begin{proof}
		If $\mathcal N\subset\mathcal N'$, then it follows from Proposition~\ref{p_trans} that $\mathcal N'$ is a completion ring for $\nu$ and $\nu' = \mathfrak C_{\mathcal N'}\nu$ and, therefore, $\nu'$ is an extension of $\nu$ by Theorem~\ref{p_Ncompl}. Let $\nu'$ be an extension of $\nu$ and $N\in\mathcal N$. Then $N\in D_\nu$ and, hence, $N\in D_{\nu'}$. Let $N'\in D_{\nu'}$ be such that $N'\subset N$. Let $B\in D_\mu$ be such that $N'\Delta B\in\mathcal N'$. Since $\mathcal N$ is a completion ring for $\mu$, the set $C = N\cap B$ belongs to $\mathcal N$. As $N'\Delta C \subset N'\Delta B$ and $N'\Delta C\in D_{\nu'}$, Lemma~\ref{l_cring} implies that $N'\Delta C\in\mathcal N'$. Hence, $\nu'(N') = \nu'(C) = \nu(C) = \nu(\varnothing) = \nu'(\varnothing)$ because $\nu'$ admits completion, $\mathcal N_{\nu'} = \mathcal N'$, and $\mathcal N_\nu = \mathcal N$ by Theorem~\ref{p_Ncompl}. Thus, $N\in \mathcal N_{\nu'}$ and, therefore, $N\in \mathcal N'$. This means that $\mathcal N\subset \mathcal N'$. If $\nu' = \nu$, then $\nu'$ and $\nu$ are extensions of each other and, hence, $\mathcal N' = \mathcal N$.   
	\end{proof}

	In Section~\ref{s5}, we shall need additional regularity conditions on $\mathcal N$ to ensure that the $\mathcal N$-completion of a $\sigma$-content remains $\sigma$-additive. We conclude this section by considering such conditions in the general setting of maps admitting completion.  
	
	\begin{definition}\label{d_reg}
		Let $\mu$ be a map such that $\varnothing\in D_\mu$. We say that a set $\mathcal N$ is $\mu$-regular if $\mathcal N\cap \hat D_\mu\subset \hat{\mathcal N}_\mu$.
	\end{definition}
	
	\begin{proposition}\label{p_reg}
		Let $\mathcal N$ be a completion ring for $\mu$. Let $A\in \Delta(D_\mu,\mathcal N)$ and $\mathcal N$ be $\mu$-regular.
		\begin{enumerate}
			\item[(i)] There is $B\in D_\mu$ such that $B\subset A$ and $A\setminus B\in \mathcal N$.
			\item[(ii)] Let $A\in \hat D_\mu$. Then there is $B\in D_\mu$ such that $A\subset B$ and $B\setminus A\in\mathcal N$.
		\end{enumerate}
	\end{proposition}
	\begin{proof}
		(i) Let $C\in D_\mu$ be such that $A\Delta C\in \mathcal N$. The sets $A\setminus C$ and $C\setminus A$ belong to $\Delta(D_\mu,\mathcal N)$ because $\Delta(D_\mu,\mathcal N)$ is a ring by Lemma~\ref{l_DeltaQN}. As $A\setminus C$ and $C\setminus A$ are subsets of $A\Delta C$, it follows from Lemma~\ref{l_DeltaQN} that $A\setminus C$ and $C\setminus A$ belong to $\mathcal N$. Since $C\setminus A\in \hat D_\mu$ and $\mathcal N$ is $\mu$-regular, we have $C\setminus A\in \hat{\mathcal N}_\mu$. Let $N\in \mathcal N_\mu$ be such that $C\setminus A\subset N$. We set $B = C\setminus N$. Then $B\in D_\mu$ and $A\setminus B = (A\setminus C)\cup (N\cap A)$. This implies that $A\setminus B\in\mathcal N$ because $N\cap A\in\mathcal N$ by Lemma~\ref{l_DeltaQN}. Since $B\setminus A = (C\setminus A)\setminus N = \varnothing$, we have $B\subset A$.
		\par\medskip\noindent
		(ii) Let $C\in D_\mu$ be such that $A\subset C$. We have $C\setminus A\in \Delta(D_\mu,\mathcal N)$ because $\Delta(D_\mu,\mathcal N)$ is a ring. By~(i), there is $B'\in D_\mu$ such that $B'\subset C\setminus A$ and the set $N = C\setminus (A\cup B')$ belongs to $\mathcal N$. Let $B = C\setminus B'$. Then $B\in D_\mu$ and $A\subset B$. Since $B\setminus A = N$, we have $B\setminus A\in\mathcal N$.
	\end{proof}
	
	The hypothesis about the $\mu$-regularity of $\mathcal N$ in Proposition~\ref{p_reg} cannot be drop\-ped. For example, let $S = \{0,1\}$ and the map $\mu$ be such that $D_\mu = \{\varnothing,S\}$, $\mu(\varnothing) = 0$, and $\mu(S) = 1$. Then $\mu$ is a positive measure. Clearly, $\mathcal N = \{\varnothing,\{0\}\}$ is a completion ring for $\mu$ and statements~(i) and~(ii) of Proposition~\ref{p_reg} do not hold for $A = \{1\}$ and $A = \{0\}$ respectively.

\section{Natural completion procedures}
\label{s4}

	Given a map $\mu$ such that $\varnothing\in D_\mu$, we define the sets $\mathcal N_0[\mu]$ and $\mathcal N[\mu]$ by equalities~(\ref{N0mu}) and~(\ref{Nmu}) respectively. Clearly, $\mathcal N[\mu]$ is a hereditary set.
	
	Let $\mu$ be a map admitting completion. By Proposition~\ref{p_Nring}, $\mathcal N_\mu$ is a ring and, hence, $\hat{\mathcal N}_\mu$ is a ring. This implies that $\mathcal N_0[\mu]$ and $\mathcal N[\mu]$ are rings. Clearly, $\mathcal N_\mu\subset \hat{\mathcal N}_\mu\subset \mathcal N[\mu]$ and $\hat{\mathcal N}_\mu\cap D_\mu\subset \mathcal N_\mu$. If $N\in \mathcal N[\mu]\cap D_\mu$, then $N\in \hat{\mathcal N}_\mu$ because $N = N\cap N$ and, hence, $N\in\mathcal N_\mu$. Thus, $\mathcal N_\mu = \hat{\mathcal N}_\mu\cap D_\mu = \mathcal N[\mu]\cap D_\mu$. Since $\mathcal N_\mu\subset \mathcal N_0[\mu]\subset \mathcal N[\mu]$, it follows that $\mathcal N_\mu = \mathcal N_0[\mu]\cap D_\mu$. If $B\in D_\mu$ and $N$ belongs to one of the sets $\hat{\mathcal N}_\mu$, $\mathcal N_0[\mu]$, and $\mathcal N[\mu]$, then $N\cap B$ obviously belongs to the same set. Thus, $\hat{\mathcal N}_\mu$, $\mathcal N_0[\mu]$, and $\mathcal N[\mu]$ are completion rings for $\mu$.  
	
	\begin{definition}
		Let $\mu$ be a map admitting completion. The $\mathcal N$-completions of $\mu$ with $\mathcal N$ equal to $\hat{\mathcal N}_\mu$, $\mathcal N_0[\mu]$, and $\mathcal N[\mu]$ are called the completion, saturation, and essential completion of $\mu$ and are denoted by $\mathfrak C\mu$, $\mathfrak S\mu$, and $\overline{\mathfrak C}\mu$ respectively.
	\end{definition}
	
	\begin{proposition}\label{p_complete}
		Let $\mu$ be a map admitting completion. Then $S[\mathfrak C\mu] = S[\mathfrak S\mu] = S[\overline{\mathfrak C}\mu] = S[\mu]$, $\mathfrak C\mu$ and $\overline{\mathfrak C}\mu$ are complete maps, $\overline{\mathfrak C}\mu$ is an extension of both $\mathfrak C\mu$ and $\mathfrak S\mu$, and 
		\[
		\mu\mbox{ is complete}\Leftrightarrow \mu = \mathfrak C\mu\Leftrightarrow \mathfrak S\mu = \overline{\mathfrak C}\mu.
		\] 
	\end{proposition}
	\begin{proof}
		Let $\nu = \mathfrak C\mu$, $\mu' = \mathfrak S\mu$, and $\nu' = \overline{\mathfrak C}\mu$.
		Since $\hat{\mathcal N}_\mu$ and $\mathcal N_0[\mu]$ are subsets of $\mathcal N[\mu]$, it follows from Proposition~\ref{p_ext} that $\nu'$ is an extension of both $\nu$ and $\mu'$. As $\bigcup \mathcal N[\mu]\subset S[\mu]$, Theorem~\ref{p_Ncompl} implies that $S[\nu] = S[\mu'] = S[\nu'] = S[\mu]$. By Theorem~\ref{p_Ncompl}, the maps $\nu$ and $\nu'$ admit completion and we have $\mathcal N_\nu = \hat{\mathcal N}_\mu$ and $\mathcal N_{\nu'} = \mathcal N[\mu]$. Since $\hat{\mathcal N}_\mu$ and $\mathcal N[\mu]$ are hereditary sets, we conclude that $\nu$ and $\nu'$ are complete. Since $\mu = \mathfrak C_{\mathcal N_\mu}\mu$ by Theorem~\ref{p_Ncompl}, it follows from Proposition~\ref{p_ext} that $\mu = \nu$ if and only if $\mathcal N_\mu = \hat{\mathcal N}_\mu$ and, hence, if and only if $\mu$ is complete. By Proposition~\ref{p_ext}, $\mu' = \nu'$ if and only if 
		\begin{equation}\label{N0=N}
		\mathcal N_0[\mu] = \mathcal N[\mu].
		\end{equation}
		If $\mu$ is complete, then $\mathcal N_\mu$ is hereditary and (\ref{N0=N}) obviously holds. Conversely, let (\ref{N0=N}) be valid. Let $N\in\mathcal N_\mu$ and $N'\subset N$. Then $N'\in\mathcal N[\mu]$ and, hence, $N'\in\mathcal N_0[\mu]$. As $N' = N'\cap N$, this implies that $N'\in\mathcal N_\mu$, i.e., $\mathcal N_\mu$ is hereditary and $\mu$ is complete. Thus, (\ref{N0=N}) is equivalent to the completeness of $\mu$.       
	\end{proof}
	
	\begin{lemma}\label{l_esscompl}
		Let $\mathcal N$ be a completion ring for $\mu$ and $\nu = \mathfrak C_{\mathcal N}\mu$. If $\mathcal N\subset \mathcal N[\mu]$, then $S[\nu] = S[\mu]$, $\mathcal N[\nu] = \mathcal N[\mu]$, and $\overline{\mathfrak C}\nu = \overline{\mathfrak C}\mu$. 
	\end{lemma}
	\begin{proof}
		By Theorem~\ref{p_Ncompl}, we have $\mathcal N_\nu = \mathcal N$ and $S[\nu] = S[\mu]\cup\bigcup\mathcal N$. Since $\mathcal N\subset \mathcal N[\mu]$ and $\mathcal N[\mu]$ is hereditary, it follows that $\hat{\mathcal N}_\nu\subset \mathcal N[\mu]$ and $S[\nu] = S[\mu]$. Let $N\in \mathcal N[\nu]$. Then $N\subset S[\nu]$ and, hence, $N\subset S[\mu]$. Let $B\in D_\mu$. Since $B\in D_\nu$, we have $N\cap B\in \hat{\mathcal N}_\nu$ and, therefore, $N\cap B\in \mathcal N[\mu]$. As $N\cap B = (N\cap B)\cap B$, we obtain $N\cap B\in \hat{\mathcal N}_\mu$. This implies that $N\in\mathcal N[\mu]$, i.e., $\mathcal N[\nu] \subset \mathcal N[\mu]$. Let $N\in \mathcal N[\mu]$. Then $N\subset S[\mu]$ and, hence, $N\subset S[\nu]$. Let $A\in D_\nu$ and let $B\in D_\mu$ be such that $A\Delta B\in \mathcal N$. Then $N\cap B\in \hat{\mathcal N}_\mu$ and, therefore, $N\cap B\in \hat{\mathcal N}_\nu$ because $\mathcal N_\mu\subset \mathcal N$. Since $N\cap A\subset (N\cap B)\cup (A\Delta B)$, we conclude that $N\cap A\in \hat{\mathcal N}_\nu$. This implies that $N\in\mathcal N[\nu]$, i.e., $\mathcal N[\mu] \subset \mathcal N[\nu]$. Thus, $\mathcal N[\nu] = \mathcal N[\mu]$. Since $\mathcal N[\nu]$ is a completion ring for $\nu$, it follows from Proposition~\ref{p_trans} that $\mathcal N[\nu]$ is a completion ring for $\mu$ and $\overline{\mathfrak C}\nu = \mathfrak C_{\mathcal N[\nu]}\mu$. Hence, $\overline{\mathfrak C}\nu = \overline{\mathfrak C}\mu$.
	\end{proof}
	
\begin{lemma}\label{l_compl1}
	Let $\mathcal N$ be a completion ring for $\mu$, $\mathcal N\subset\mathcal N_0[\mu]$, and $\nu = \mathfrak C_{\mathcal N}\mu$. Then 
	$\mathcal N[\nu] = \mathcal N[\mu]$ and 
	\begin{equation}\label{N0}
	\mathcal N_0[\nu] = \{N\in \mathcal N_0[\mu]: N\cap N'\in \mathcal N\mbox{ for every }N'\in\mathcal N\}.
	\end{equation}
\end{lemma}	 
\begin{proof}
	Since $\mathcal N_0[\mu]\subset \mathcal N[\mu]$, we have $\mathcal N\subset\mathcal N[\mu]$ and Lemma~\ref{l_esscompl} implies that $\mathcal N[\nu] = \mathcal N[\mu]$ and $S[\nu]=S[\mu]$. Let $\mathcal K$ denote the set in the right-hand side of~(\ref{N0}). Let $N\in\mathcal K$. Since $N\in\mathcal N_0[\mu]$, we have $N\subset S[\mu]$ and, hence, $N\subset S[\nu]$. Let $A\in D_\nu$ and let $B\in D_\mu$ and $N'\in \mathcal N$ be such that $A = B\Delta N'$. By~(\ref{distr}), we have $N\cap A = (N\cap B)\Delta(N\cap N')$. It follows that $N\cap A\in\mathcal N$ because $N\cap B\in\mathcal N_\mu$, $N\cap N'\in\mathcal N$, and $\mathcal N_\mu\subset\mathcal N$. Since $\mathcal N_\nu = \mathcal N$ by Theorem~\ref{p_Ncompl}, we conclude that $N\in\mathcal N_0[\nu]$, i.e., $\mathcal K\subset \mathcal N_0[\nu]$. Let $N\in\mathcal N_0[\nu]$. Then $N\subset S[\nu]$ and, hence, $N\subset S[\mu]$. If $B\in D_\mu$, then $N\cap B\in \mathcal N$ because $D_\mu\subset D_\nu$ and, therefore, $N\cap B\in \mathcal N_0[\mu]$. As $N\cap B = (N\cap B)\cap B$, it follows that $N\cap B\in \mathcal N_\mu$ and, hence, $N\in\mathcal N_0[\mu]$. Since $\mathcal N = \mathcal N_\nu$, we have $N\cap N'\in \mathcal N$ for every $N'\in\mathcal N$. Thus, $N\in\mathcal K$, i.e., $\mathcal N_0[\nu]\subset\mathcal K$.   
\end{proof}

	\begin{proposition}\label{p_nat}
		Let $\mu$ be a map admitting completion. Then $\mathfrak C\mathfrak C\mu = \mathfrak C\mu$, $\mathfrak S\mathfrak S\mu = \mathfrak S\mu$, $\overline{\mathfrak C}\mu$ is an extension of $\mathfrak C \mathfrak S\mu$, and
		\begin{equation}\nonumber
		\overline{\mathfrak C}\mu = \overline{\mathfrak C}\,\overline{\mathfrak C}\mu = \overline{\mathfrak C}\mathfrak C \mu = \overline{\mathfrak C}\mathfrak S \mu = \mathfrak C \overline{\mathfrak C}\mu = \mathfrak S \overline{\mathfrak C}\mu = \mathfrak S \mathfrak C\mu. 
		\end{equation}
		The equality $\overline{\mathfrak C}\mu = \mathfrak C\mathfrak S\mu$ holds if and only if $\mathcal N[\mu] = \widehat{\mathcal N_0[\mu]}$.
	\end{proposition}
	\begin{proof}
		Since $\hat{\mathcal N}_\mu$ and $\mathcal N_0[\mu]$ are subsets of $\mathcal N[\mu]$, Lemma~\ref{l_esscompl} implies that $\overline{\mathfrak C}\mu = \overline{\mathfrak C}\,\overline{\mathfrak C}\mu = \overline{\mathfrak C}\mathfrak C \mu = \overline{\mathfrak C}\mathfrak S \mu$. This implies that $\overline{\mathfrak C}\mathfrak C\mathfrak S\mu = \overline{\mathfrak C}\mu$ and, hence, $\overline{\mathfrak C}\mu$ is an extension of $\mathfrak C\mathfrak S\mu$. By Proposition~\ref{p_complete}, $\mathfrak C\mu$ and $\overline{\mathfrak C}\mu$ are complete. Applying the same proposition, we obtain $\mathfrak C\mathfrak C\mu = \mathfrak C\mu$, $\mathfrak C \overline{\mathfrak C}\mu = \overline{\mathfrak C}\mu$, $\mathfrak S \mathfrak C\mu = \overline{\mathfrak C}\mathfrak C\mu$, and $\mathfrak S\overline{\mathfrak C}\mu = \overline{\mathfrak C}\,\overline{\mathfrak C}\mu$. Since $\mathcal N_0[\mathfrak S\mu]$ is a completion ring for $\mathfrak S\mu$, it follows from Proposition~\ref{p_trans} that $\mathcal N_0[\mathfrak S\mu]$ is a completion ring for $\mu$ and $\mathfrak S\mathfrak S\mu = \mathfrak C_{\mathcal N_0[\mathfrak S\mu]}\mu$. By Lemma~\ref{l_compl1}, we have $\mathcal N_0[\mathfrak S\mu] = \mathcal N_0[\mu]$ and, hence, $\mathfrak S\mathfrak S\mu = \mathfrak S\mu$. Let $\mathcal N = \mathcal N_{\mathfrak S\mu}$. As $\hat{\mathcal N}$ is a completion ring for $\mathfrak S\mu$, it follows from Proposition~\ref{p_trans} that $\hat{\mathcal N}$ is a completion ring for $\mu$ and $\mathfrak C\mathfrak S\mu = \mathfrak C_{\hat{\mathcal N}}\mu$. By Proposition~\ref{p_ext}, we conclude that $\overline{\mathfrak C}\mu = \mathfrak C\mathfrak S\mu$ if and only if $\mathcal N[\mu] = \hat{\mathcal N}$. It remains to note that $\mathcal N = \mathcal N_0[\mu]$ by Theorem~\ref{p_Ncompl}.
	\end{proof}

	In general, $\overline{\mathfrak C}\mu$ and $\mathfrak C \mathfrak S\mu$ need not coincide (see Example~\ref{e1}). A sufficient condition for the equality $\overline{\mathfrak C}\mu = \mathfrak C \mathfrak S\mu$ will be given in Section~\ref{s6}.
		
\section{Completion of contents}
\label{s5}
	
	\begin{lemma}\label{l_partition}
		Let $\mathcal N$ be a completion ring for $\mu$ and $\fm{\mb A}{I}$ be a countable disjoint family. Let $\mb C\colon I\to D_\mu$ be such that $\mb A(i)\Delta\mb C(i)\in \mathcal N$ for every $i\in I$. Then there is a disjoint family $\mb B\colon I\to D_\mu$ such that $\bigcup_{i\in I}\mb B(i) = \bigcup_{i\in I}\mb C(i)$ and $\mb A(i)\Delta\mb B(i)\in \mathcal N$ for every $i\in I$.
	\end{lemma}
	\begin{proof}
		We can assume that $I$ is linearly ordered in such a way that the set $\{j\in I: j\leq i\}$ is finite for every $i\in I$ (such an order exists because $I$ is countable). Let $\fm{\mb B}{I}$ be such that $\mb B(i) = \mb C(i)\setminus\bigcup_{j<i} \mb C(j)$ for every $i\in I$. Since $D_\mu$ is a ring, $\mb B(i)\in D_\mu$ for all $i\in I$ and we obviously have $\bigcup_{i\in I}\mb B(i) = \bigcup_{i\in I}\mb C(i)$. If $i\neq j$, then 
		\[
		\mb C(i)\cap \mb C(j)\subset (\mb A(i)\Delta \mb C(i))\cup (\mb A(j)\Delta \mb C(j))
		\]
		because $\mb A(i)\cap \mb A(j)=\varnothing$. In view of~(\ref{trans}), we obtain
		\begin{multline}\label{incl}
		\mb A(i)\Delta \mb B(i)\subset (\mb A(i)\Delta \mb C(i))\cup (\mb C(i)\Delta \mb B(i)) \\ = (\mb A(i)\Delta \mb C(i))\cup \bigcup_{j<i} (\mb C(i)\cap \mb C(j)) \subset \bigcup_{j\leq i} \mb A(j)\Delta \mb C(j),\quad i\in I.
		\end{multline}
		Lemma~\ref{l_DeltaQN} implies that the set $\mathcal K = \Delta(D_\mu,\mathcal N)$ is a ring. 
		Since $\mb A(i)\in\mathcal K$, it follows that $\mb A(i)\Delta \mb B(i)\in\mathcal K$ for every $i\in I$.
		As $\mathcal N$ is a ring, the set in the right-hand side of~(\ref{incl}) belongs to $\mathcal N$ and, hence, $\mb A(i)\Delta \mb B(i)\in\mathcal N$ for every $i\in I$ by Lemma~\ref{l_DeltaQN}.
	\end{proof}
	
	\begin{lemma}\label{l_part}
		Let $\mathcal N$ be a completion ring for $\mu$ and $\nu = \mathfrak C_{\mathcal N}\mu$. Let $\mb A\colon I\to D_\nu$ be a countable disjoint family. Then there is a disjoint family $\mb B\colon I\to D_\mu$ such that $\mb A(i)\Delta\mb B(i)\in \mathcal N$ for every $i\in I$.
	\end{lemma}
	\begin{proof}
		The statement follows immediately from Lemma~\ref{l_partition}.
	\end{proof}

	\begin{theorem}\label{ll10}
		Let $\mathfrak A$ be an Abelian group and $\mu$ be an $\mathfrak A$-valued content. Then $\mathcal N_\mu$ is a ring and $\mu$ admits completion. The map $\mu$ is complete if and only if $\mathcal N_\mu$ is a hereditary set. If $\mathcal N$ is a completion ring for $\mu$, then $\mathfrak{C}_{\mathcal N}\mu$ is an $\mathfrak A$-valued content. In particular, $\mathfrak C\mu$, $\mathfrak S\mu$, and $\overline{\mathfrak C}\mu$ are $\mathfrak A$-valued contents.
	\end{theorem}
	\begin{proof}
		Let $B_1, B_2\in D_\mu$ be such that $B_1\Delta B_2\in\mathcal N_\mu$. Then $B_1\setminus B_2$, $B_1\cap B_2$ and $B_2\setminus B_1$ are pairwise disjoint elements of $D_\mu$. Since $B_1\setminus B_2$ and $B_2\setminus B_1$ are both subsets of $B_1\Delta B_2$, we have $\mu(B_1\setminus B_2) = \mu(B_2\setminus B_1) = 0$ and it follows from the additivity of $\mu$ that $\mu(B_1) = \mu(B_1\cap B_2) = \mu(B_2)$. Hence, $\mu$ admits completion. 
		By Proposition~\ref{p_Nring}, it follows that $\mathcal N_\mu$ is a ring, and Definition~\ref{dd7} implies that $\mu$ is complete if and only if $\mathcal N_\mu$ is hereditary. 
		Let $\mathcal N$ be a completion ring for $\mu$ and $\nu = \mathfrak{C}_{\mathcal N}\mu$. Let $A\in D_\nu$ and $\mb A\colon I\to D_\nu$ be a finite partition of $A$. By Lemma~\ref{l_part}, there is a disjoint family $\mb B\colon I\to D_\mu$ such that $\mb A(i)\Delta\mb B(i)\in\mathcal N$ for every $i\in I$. As $D_\mu$ is a ring, the set $B = \bigcup_{i\in I}\mb B(i)$ belongs to $D_\mu$. Since $A\Delta B\in D_\nu$, it follows from~(\ref{0}) and Lemma~\ref{l_cring} that $A\Delta B\in\mathcal N$. Thus, $\nu(A) = \mu(B)$ and $\nu(\mb A(i)) = \mu(\mb B(i))$ for every $i\in I$. Since $\mu(B) = \sum_{i\in I}\mu(\mb B(i))$ by the additivity of $\mu$, we conclude that $\nu(A) = \sum_{i\in I}\nu(\mb A(i))$, i.e., $\nu$ is an $\mathfrak A$-valued additive function. Since $D_\nu$ is a ring by Theorem~\ref{p_Ncompl}, we conclude that $\nu$ is an $\mathfrak A$-valued content.    
	\end{proof}
	
	In general, $\mathcal N$-completions of $\sigma$-contents may be not $\sigma$-additive (see Examples~\ref{e4} and~\ref{e5}). The next theorem shows, however, that $\sigma$-additivity is always preserved if $\mathcal N$ is regular in the sense of Definition~\ref{d_reg}.  	
	
	\begin{theorem}\label{p_scontent}
		Let $\mathfrak A$ be a HTAG, $\mu$ be an $\mathfrak A$-valued $\sigma$-content, and $\mathcal N$ be a completion ring for $\mu$. If $\mathcal N$ is $\mu$-regular, then $\mathfrak{C}_{\mathcal N}\mu$ is an $\mathfrak A$-valued $\sigma$-content.
	\end{theorem}
	\begin{proof}
		Let $\nu = \mathfrak{C}_{\mathcal N}\mu$, $A\in D_\nu$, and $\mb A\colon I\to D_\nu$ be a countable partition of $A$. By Proposition~\ref{p_reg}(i), there is $B\in D_\mu$ such that $B\subset A$ and $A\setminus B\in\mathcal N$. Let $\fm{\mb A'}{I}$ be such that $\mb A'(i) = \mb A(i)\cap B$ for every $i\in I$. Then $\mb A'(i)\in D_\nu\cap \hat D_\mu$ for all $i\in I$. By Proposition~\ref{p_reg}(ii), there is $\mb C'\colon I\to D_\mu$ such that $\mb A'(i)\subset \mb C'(i)$ and $\mb C'(i)\setminus \mb A'(i)\in\mathcal N$ for every $i\in I$. Let $\fm{\mb C}{I}$ be such that $\mb C(i) = \mb C'(i)\cap B$ for every $i\in I$. As $\mb A'(i)\subset \mb C(i)\subset B$ for all $i\in I$ and $B = \bigcup_{i\in I}\mb A'(i)$, we have $B = \bigcup_{i\in I}\mb C(i)$. Since $\mb A(i)\setminus \mb A'(i) = \mb A(i)\setminus B\subset A\setminus B$, it follows from~(\ref{trans}) that
		\[
		\mb A(i)\Delta \mb C(i) \subset (\mb A(i)\setminus \mb A'(i))\cup (\mb C(i)\setminus \mb A'(i))\subset (A\setminus B)\cup (\mb C'(i)\setminus \mb A'(i))
		\] 
		for every $i\in I$. As $\mb A(i)\Delta \mb C(i)\in D_\nu$, Lemma~\ref{l_cring} implies that $\mb A(i)\Delta \mb C(i)\in \mathcal N$ for every $i\in I$. By Lemma~\ref{l_partition}, there is a disjoint partition $\mb B\colon I\to D_\mu$ of $B$ such that $\mb A(i)\Delta\mb B(i)\in\mathcal N$ for every $i\in I$. By the $\sigma$-additivity of $\mu$, we have $\mu(B) = \sum_{i\in I}\mu(\mb B(i))$. Since $\nu(A) = \mu(B)$ and $\nu(\mb A(i)) = \mu(\mb B(i))$ for every $i\in I$, we conclude that $\nu(A) = \sum_{i\in I}\nu(\mb A(i))$. Thus, $\nu$ is an $\mathfrak A$-valued $\sigma$-additive function. Since $D_\nu$ is a ring by Theorem~\ref{p_Ncompl}, we conclude that $\nu$ is an $\mathfrak A$-valued $\sigma$-content.
	\end{proof}

	\begin{proposition}\label{p_reg1}
		Let $\mu$ be a map, $\varnothing\in D_\mu$, and $\mathcal N\subset \mathcal N[\mu]$. Then $\mathcal N$ is $\mu$-regular.
	\end{proposition}
	\begin{proof}
		Let $N\in \mathcal N\cap\hat D_\mu$ and $B\in D_\mu$ be such that $N\subset B$. Since $N\in\mathcal N[\mu]$ and $N = N\cap B$, we have $N\in\hat{\mathcal N}_\mu$. Thus, $\mathcal N\cap\hat D_\mu\subset \hat{\mathcal N}_\mu$, i.e., $\mathcal N$ is $\mu$-regular.
	\end{proof}

	\begin{proposition}
		Let $\mathfrak A$ be a HTAG and $\mu$ be an $\mathfrak A$-valued $\sigma$-content. Then $\mathfrak C\mu$, $\mathfrak S\mu$, and $\overline{\mathfrak C}\mu$ are $\mathfrak A$-valued $\sigma$-contents. 
	\end{proposition}
	\begin{proof}
		The statement follows at once from Theorem~\ref{p_scontent} and Proposition~\ref{p_reg1}.  
	\end{proof}
	
	We conclude this section by discussing how the procedure of $\mathcal N$-completion is related to the $\mu$-topology induced by a content $\mu$ on its domain. We refer the reader to~\cite{Weber} for a detailed treatment of $\mu$-topologies and, more generally, Fr\'echet--Nikodym topologies on Boolean rings. Such topologies turn out to be a useful tool for proving extension and decomposition results for group-valued contents~\cite{Weber} and establishing the compactness properties of the range of a content in a locally convex space~\cite{UrbinatiWeber}. After giving basic definitions, we shall show that a content taking values in a HTAG and all its $\mathcal N$-completions give rise to essentially the same topological Boolean rings (see Proposition~\ref{p_assH} below).
	
	Recall that $R$ is called a Boolean ring if it is an (algebraic) ring such that $x^2 = x$ for every $x\in R$. It is assumed that $R$ is endowed with the order such that, for every $x,y\in R$, we have $x\leq y$ if and only if $xy = x$. Let $\mathfrak A$ be an Abelian group. A map $\mu\colon R\to\mathfrak A$ is called an $\mathfrak A$-valued $R$-content if $\mu(x+y) = \mu(x)+\mu(y)$ for every $x,y\in R$ such that $xy=0$.
	
	Let $R$ be a Boolean ring, $\mathfrak A$ be a topological Abelian group, and $\mu$ be an $\mathfrak A$-valued $R$-content. For every set $U$, we define
	\begin{equation}\label{VRmuU}
	V_{R,\mu}(U) = \{x\in R: \mu(y)\in U \mbox{ for every }y\in R\mbox{ such that }y\leq x\}.
	\end{equation}
	Let $\mathcal V$ be the set of all sets $V_{R,\mu}(U)$, where $U$ is a neighbourhood of the origin in $\mathfrak A$. It is easy to see that $\mathcal V$ is a basis of neighbourhoods of the origin for a uniquely determined group topology on $R$. The latter is called the $\mu$-topology for $(R,\mathfrak A)$ and actually makes $R$ into a topological ring, which is denoted by $R\langle\mu,\mathfrak A\rangle$. It follows immediately from the definitions that
	\begin{enumerate}
		\item[(a)] if $\mathfrak A$ is Hausdorff, then the closure of $\{0\}$ in $R\langle\mu,\mathfrak A\rangle$ is equal to $V_{R,\mu}(\{0_{\mathfrak A}\})$,
		\item[(b)] if $R'$ is a subring of $R$, then $V_{R,\mu}(U)\cap R'\subset V_{R',\mu|_{R'}}(U)$ for every set $U$.
	\end{enumerate}
	
	Let $\mathcal Q$ be a ring of sets. In view of~(\ref{assoc}) and~(\ref{distr}), $\mathcal Q$ becomes a Boolean ring if we define $A+B = A\Delta B$ and $AB = A\cap B$ for every $A,B\in\mathcal Q$. This Boolean ring will be denoted by $\check{\mathcal Q}$. Let $\mathfrak A$ be an Abelian group. A map $\mu$ is an $\mathfrak A$-valued content if and only if $D_\mu$ is a ring of sets and $\mu$ is an $\mathfrak A$-valued $\check D_\mu$-content.
	
	If $R$ is a topological ring, then the closure $\overline{\{0\}}$ of $\{0\}$ in $R$ is a closed ideal of $R$. The quotient ring $R/\overline{\{0\}}$ is Hausdorff and is called the Hausdorff ring associated with $R$.
	
	The next proposition shows that the topological Boolean rings determined by a content taking values in a HTAG and by its $\mathcal N$-completions have the same associated Hausdorff rings.
	
	\begin{proposition}\label{p_assH}
		Let $\mathfrak A$ be a HTAG, $\mu$ be an $\mathfrak A$-valued content, $\mathcal N$ be a completion ring for $\mu$, and $\nu = \mathfrak C_{\mathcal N}\mu$. Let $R = \check D_\nu\langle\nu,\mathfrak A\rangle$ and $R' = \check D_\mu\langle\mu,\mathfrak A\rangle$. Then $R'$ is a dense subring of $R$, $\overline{\{0\}}{}^R = \mathcal N$, $\overline{\{0\}}{}^{R'} = \mathcal N_\mu$, and there is a unique topological ring isomorphism $l\colon R'/\mathcal N_\mu\to R/\mathcal N$ such that $l\circ j' =j|_{R'}$, where $j\colon R\to R/\mathcal N$ and $j'\colon R'\to R'/\mathcal N_\mu$ are canonical maps. 
	\end{proposition}
	
	To prove Proposition~\ref{p_assH}, we need two auxiliary statements.
	
	\begin{lemma}\label{l_assH}
		Let $R$ be a topological Abelian group, $\mathcal N = \overline{\{0\}}$, and $U\subset R$ be an open set. If $x\in U$, then $x+\mathcal N\subset U$.
	\end{lemma}
	\begin{proof}
		Let $x\in U$ and $y\in\mathcal N$. Since $-y\in\mathcal N$ and $U-x-y$ is a neighbourhood of $-y$, we have $0\in U-x-y$ and, hence, $x+y\in U$. Thus, $x+\mathcal N\subset U$. 
	\end{proof}
	
	\begin{lemma}\label{l_assH1}
		Let $R$ be a topological ring, $R'$ be its topological subring, $\mathcal N = \overline{\{0\}}{}^R$, and $\mathcal N' = \overline{\{0\}}{}^{R'}$. Suppose $R = R' + \mathcal N$. Then $R'$ is dense in $R$ and there is a unique topological ring isomorphism $l\colon R'/\mathcal N'\to R/\mathcal N$ such that $l\circ j' = j|_{R'}$, where $j\colon R\to R/\mathcal N$ and $j'\colon R'\to R'/\mathcal N'$ are canonical maps. 
	\end{lemma}
	\begin{proof}
		Let $k = j|_{R'}$. Since $R = R' + \mathcal N$ and $\mathcal N' = \mathcal N\cap R'$, there is a unique ring isomorphism $l\colon R'/\mathcal N'\to R/\mathcal N$ such that $l\circ j' = k$. By the definition of quotient topology, $l$ is continuous. Let $V\subset R'/\mathcal N'$ be an open set and $W = j^{\prime-1}(V)$. Since $W$ is open in $R'$, there is an open subset $U$ of $R$ such that $W = U\cap R'$. Let $x\in U$. Then $x-y\in \mathcal N$ for some $y\in R'$. As $y\in U$ by Lemma~\ref{l_assH}, we have $y\in W$. This implies that $j(x) \in k(W)$ because $j(x) = k(y)$. Hence, $j(U) = k(W)$. Since $j$ is open as a quotient map of a topological group, it follows that the set $l(V) = k(W)$ is open. Thus, the map $l$ is open and, hence, is a homeomorphism. For every $x\in R'$, we have $\overline{R'}\supset \overline{\{x\}}{}^R = x + \mathcal N$. This implies that $\overline{R'} = R$, i.e., $R'$ is dense in $R$.   
	\end{proof}
	
	\begin{proof}[Proof of Proposition~$\ref{p_assH}$]
		Since $\mathcal N = \mathcal N_\nu$ by Theorem~\ref{p_Ncompl}, the equalities $\overline{\{0\}}{}^R = \mathcal N$ and $\overline{\{0\}}{}^{R'} = \mathcal N_\mu$ follow from~(\ref{VRmuU}) and statement~(a). As $D_\nu = \Delta(D_\mu,\mathcal N)$, we have $R = R' + \mathcal N$. We claim that
		\begin{equation}\label{V=V}
		V_{R',\mu}(U) = V_{R,\nu}(U)\cap D_\mu 
		\end{equation}
		for every set $U$. Let $V = V_{R,\nu}(U)$, $V' = V_{R',\mu}(U)$, and $A'\in V'$. Let $A\in D_\nu$ be such that $A\subset A'$. Then $A = B\Delta N$ for some $B\in D_\mu$ and $N\in \mathcal N$. By~(\ref{distr}), it follows that $A = A\cap A' = B'\Delta N'$, where $B' = B\cap A'$ and $N' = N\cap A'$. Since $B'\in D_\mu$ and $B'\subset A'$, we have $\mu(B')\in U$. This implies that $\nu(A)\in U$ because $N'\in\mathcal N$ and, hence, $\nu(A) = \mu(B')$. This means that $A'\in V$. We thus have $V'\subset V$, whence (\ref{V=V}) follows by statement~(b). Equality~(\ref{V=V}) ensures that $R'$ is a topological subring of $R$. The density of $R'$ in $R$ and the existence and uniqueness of $l$ therefore follow from Lemma~\ref{l_assH1}.  
	\end{proof}

\section{Completion of measures}
\label{s6}
	
	As shown by Example~\ref{e4}, the $\mathcal N$-completion of a positive measure $\mu$ may fail to be defined on a $\delta$-ring even if $D_\mu$ is a $\sigma$-ring. It turns out that conditions ensuring that the domain of the $\mathcal N$-completion of an $\mathfrak A$-valued measure is a $\delta$-ring can be conveniently formulated without any reference to the group $\mathfrak A$. This is done using the next definition that captures the typical structure of null sets of measures. 
	
	\begin{definition}\label{d_measurelike}
		A map $\mu$ is called measure-like if it admits completion, $D_\mu$ is a $\delta$-ring, and $\sigma(\mathcal N_\mu)\cap D_\mu \subset \mathcal N_\mu$. 
	\end{definition}
	
	If $\mu$ is a measure-like map, then
	\begin{equation}\label{mlike}
	\sigma(\mathcal N_\mu)\cap \hat D_\mu = \mathcal N_\mu.
	\end{equation}
	Indeed, if $N\in \sigma(\mathcal N_\mu)\cap \hat D_\mu$, then $N\in D_\mu$ by Lemma~\ref{ll1}(iii) and, hence, $N\in \mathcal N_\mu$. Since every $N\in \mathcal N_\mu$ obviously belongs to $\sigma(\mathcal N_\mu)\cap \hat D_\mu$, we obtain~(\ref{mlike}).
	
	\begin{definition}\label{d_compat}
		Let $\mu$ be a map. A set $\mathcal N$ is called $\mu$-compatible if it is a $\delta$-ring and $\sigma(\mathcal N)\cap \hat D_\mu\subset \mathcal N$.
	\end{definition}
	
	If $\mu$ is a measure-like map, then $\mathcal N_\mu$ is $\mu$-compatible by~(\ref{mlike}).
	
	\begin{proposition}\label{p_measurelike}
		Let $\mathcal N$ be a completion ring for $\mu$ and $\nu = \mathfrak C_{\mathcal N}\mu$. If $\nu$ is measure-like, then $\mathcal N$ is $\mu$-compatible. If $D_\mu$ is a $\delta$-ring and $\mathcal N$ is $\mu$-compatible, then both $\mu$ and $\nu$ are measure-like.
	\end{proposition}
	\begin{proof}
		By Theorem~\ref{p_Ncompl}, we have
		\begin{equation}\label{NN}
		\mathcal N = \mathcal N_\nu.
		\end{equation}
		Let $\nu$ be measure-like. Then $\mathcal N_\nu$ is a $\delta$-ring by Proposition~\ref{p_Nring}, and it follows from~(\ref{mlike}) and the inclusion $D_\mu\subset D_\nu$ that $\sigma(\mathcal N_\nu)\cap \hat D_\mu\subset \mathcal N_\nu$. In view of~(\ref{NN}), we conclude that $\mathcal N$ is $\mu$-compatible. 
		
		Let $D_\mu$ be a $\delta$-ring and $\mathcal N$ be $\mu$-compatible. Then $\mu$ is measure-like because
		\[
		\sigma(\mathcal N_\mu)\cap D_\mu \subset (\sigma(\mathcal N)\cap\hat D_\mu)\cap D_\mu\subset \mathcal N\cap D_\mu = \mathcal N_\mu
		\]
		by Definitions~\ref{d_cring} and~\ref{d_compat}. To prove that $\nu$ is measure-like, we first show that 
		\begin{equation}\label{sigmacap1}
		\sigma(\mathcal N)\cap\hat D_\nu\subset \mathcal N.
		\end{equation}
		Let $N \in \sigma(\mathcal N)\cap\hat D_\nu$ and let $A\in D_\nu$ be such that $N\subset A$. Let $B\in D_\mu$ be such that $A\Delta B\in\mathcal N$. We have $N\subset N'$, where $N' = (N\cap B)\cup (A\Delta B)$. Since $N\cap B\in \sigma(\mathcal N)$ by Lemma~\ref{l_NcapB}, it follows from the $\mu$-compatibility of $\mathcal N$ that $N\cap B\in\mathcal N$ and, hence, $N'\in \mathcal N$.
		Lemma~\ref{ll1}(iii) now ensures that $N\in\mathcal N$ and (\ref{sigmacap1}) is proved.
		
		We now verify that $D_\nu$ is a $\delta$-ring. Let $\mb A\colon I\to D_\nu$ be a nonempty countable family and $A = \bigcap_{i\in I}\mb A(i)$. Let $\mb B\colon I\to D_\mu$ be such that $\mb A(i)\Delta\mb B(i)\in\mathcal N$ for every $i\in I$ and let $B = \bigcap_{i\in I}\mb B(i)$. Since $D_\mu$ is a $\delta$-ring, we have $B\in D_\mu$. Let $N = A\Delta B$. By~(\ref{01}), $N\subset N'$, where $N' = \bigcup_{i\in I}\mb A(i)\Delta \mb B(i)$. Since $N\in \sigma(D_\nu)$ and $N'\in \sigma(\mathcal N)$, Lemmas~\ref{l_NcapB} and~\ref{l_cring} imply that $N\in \sigma(\mathcal N)$. We have $N\in\hat D_\nu$ because $N\subset A\cup B\subset \mb A(i)\cup B$ for every $i\in I$. By~(\ref{sigmacap1}), we conclude that $N\in \mathcal N$ and, hence, $A\in D_\nu$. Thus, $D_\nu$ is a $\delta$-ring. Since $\nu$ admits completion by Theorem~\ref{p_Ncompl}, it follows from~(\ref{NN}) and~(\ref{sigmacap1}) that $\nu$ is measure-like.   
	\end{proof}
	
	\begin{corollary}\label{c_sigma}
		Let $\mathcal N$ be a completion ring for $\mu$, $\mathcal N$ be a $\sigma$-ring, $D_\mu$ be a $\delta$-ring, and $\nu = \mathfrak C_{\mathcal N}\mu$. Then $\mu$ and $\nu$ are measure-like. If $D_\mu$ is a $\sigma$-ring, then so is $D_\nu$. 
	\end{corollary}
	\begin{proof}
		By Definition~\ref{d_compat}, $\mathcal N$ is $\mu$-compatible and, therefore, $\mu$ and $\nu$ are measure-like by Proposition~\ref{p_measurelike}. In particular, $D_\nu$ is a $\delta$-ring. Suppose $D_\mu$ is a $\sigma$-ring. Let $\mb A\colon I\to D_\nu$ be a countable family and let $\mb B\colon I\to D_\mu$ and $\mb N\colon I\to\mathcal N$ be such that $\mb A(i) = \mb B(i)\Delta\mb N(i)$ for every $i\in I$. Then $\mb A(i)\subset B\cup N$ for every $i\in I$, where $B = \bigcup_{i\in I}\mb B(i)$ and $N = \bigcup_{i\in I}\mb N(i)$. 
		Since $D_\mu$ and $\mathcal N$ are $\sigma$-rings, we have $B\in D_\mu$ and $N\in\mathcal N$, whence $B\cup N\in D_\nu$. By Lemma~\ref{ll1}(i), we conclude that the set $\bigcup_{i\in I}\mb A(i)$ belongs to $D_\nu$. Thus, $D_\nu$ is a $\sigma$-ring.	
	\end{proof}
	
	The next statement gives a simple sufficient condition for the $\mu$-compatibility of a $\delta$-ring $\mathcal N$.
	
	\begin{proposition}\label{p_comp}
		Let $\mu$ be a measure-like map and $\mathcal N$ be a $\delta$-ring such that $\mathcal N_\mu\subset\mathcal N$. If $\mathcal N$ is $\mu$-regular, then $\mathcal N$ is $\mu$-compatible.
	\end{proposition}
	\begin{proof}
		Let $N \in \sigma(\mathcal N)\cap\hat D_\mu$ and let $C\in D_\mu$ be such that $N\subset C$. Since $\mathcal N$ is a $\delta$-ring, Lemma~\ref{ll1}(ii) implies that $N = \bigcup_{i\in I}\mb N(i)$, where $\mb N\colon I\to \mathcal N$ is a countable family. For every $i\in I$, the set $\mb N(i)$ belongs to $\mathcal N\cap \hat D_\mu$ and, hence, to $\hat{\mathcal N}_\mu$ by the $\mu$-regularity of $\mathcal N$. Let $\mb N'\colon I\to\mathcal N_\mu$ be such that $\mb N(i)\subset \mb N'(i)\subset C$ for every $i\in I$. Let $N' = \bigcup_{i\in I} \mb N'(i)$. Then $N'\subset C$ and, hence, $N'\in \hat D_\mu$. Since $N'\in \sigma(\mathcal N_\mu)$, it follows from~(\ref{mlike}) that $N'\in \mathcal N_\mu$ and, therefore, $N'\in\mathcal N$. By Lemma~\ref{ll1}(iii), we conclude that $N\in\mathcal N$ because $N\subset N'$ and $N\in\sigma(\mathcal N)$. This proves that $\sigma(\mathcal N)\cap\hat D_\mu\subset \mathcal N$ and, hence, $\mathcal N$ is $\mu$-compatible. 
	\end{proof}
	
	\begin{corollary}\label{c_mucomp}
		If $\mu$ is a measure-like map, then the sets $\hat{\mathcal N}_\mu$, $\mathcal N_0[\mu]$, and $\mathcal N[\mu]$ are $\mu$-compatible.
	\end{corollary}
	\begin{proof}
		Since $\mathcal N_\mu$ is a $\delta$-ring by Proposition~\ref{p_Nring}, it follows from~(\ref{N0mu}) that $\mathcal N_0[\mu]$ is a $\delta$-ring. The sets $\hat{\mathcal N}_\mu$ and $\mathcal N[\mu]$ are hereditary rings and, hence, are $\delta$-rings. The required statement therefore follows from Propositions~\ref{p_reg1} and~\ref{p_comp}.
	\end{proof}

	\begin{corollary}
		If $\mu$ is a measure-like map, then so are $\mathfrak C\mu$, $\mathfrak S\mu$, and $\overline{\mathfrak C}\mu$.
	\end{corollary}
	\begin{proof}
		The statement follows at once from Proposition~\ref{p_measurelike} and Corollary~\ref{c_mucomp}. 
	\end{proof}
	
	We now use the results concerning measure-like maps to obtain a criterion for the $\mathcal N$-completion of a measure to be a measure.
	
	\begin{proposition}\label{p_mlike}
		Let $\mathfrak A$ be a HTAG and $\mu$ be an $\mathfrak A$-valued measure. Then $\mu$ is measure-like. 
	\end{proposition}
	\begin{proof}
		By Theorem~\ref{ll10}, $\mu$ admits completion, and we only need to show that $\sigma(\mathcal N_\mu)\cap D_\mu \subset \mathcal N_\mu$. Let $N \in \sigma(\mathcal N_\mu)\cap D_\mu$. As $\mathcal N_\mu$ is a $\delta$-ring by Proposition~\ref{p_Nring}, Lemma~\ref{ll1}(ii) implies that there is a countable partition $\mb N\colon I\to \mathcal N_\mu$ of $N$. Let $N'\in D_\mu$ be such that $N'\subset N$. Since $\mb N(i)\in\mathcal N_\mu$, we have $\mu(\mb N(i)\cap N') = 0$ for every $i\in I$ and, hence, $\mu(N') = 0$ by the $\sigma$-additivity of $\mu$. Thus, $N\in\mathcal N_\mu$.  
	\end{proof}

	\begin{theorem}\label{t_meas}
		Let $\mathfrak A$ be a HTAG, $\mu$ be an $\mathfrak A$-valued measure, and $\mathcal N$ be a completion ring for $\mu$. Then $\mathfrak C_{\mathcal N}\mu$ is an $\mathfrak A$-valued measure if and only if $\mathcal N$ is $\mu$-compatible.
	\end{theorem}
	\begin{proof}
		Let $\nu = \mathfrak C_{\mathcal N}\mu$. If $\nu$ is an $\mathfrak A$-valued measure, then $\nu$ is measure-like by Proposition~\ref{p_mlike} and, therefore, $\mathcal N$ is $\mu$-compatible by Proposition~\ref{p_measurelike}. Let $\mathcal N$ be $\mu$-compatible. Then $\nu$ is measure-like by Proposition~\ref{p_measurelike} and, hence, $D_\nu$ is a $\delta$-ring. It remains to show that $\nu$ is $\sigma$-additive. Let $A\in D_\nu$ and $\mb A\colon I\to D_\nu$ be a countable partition of $A$. By Lemma~\ref{l_part}, there is a disjoint family $\mb C\colon I\to D_\mu$ such that $\mb A(i)\Delta\mb C(i)\in\mathcal N$ for every $i\in I$. Let $C\in D_\mu$ be such that $A\Delta C\in\mathcal N$. 
		Let $\fm{\mb B}{I}$ be such that $\mb B(i) = C\cap \mb C(i)$ for every $i\in I$. Since
		\[
		\mb C(i)\setminus \mb B(i) = \mb C(i)\setminus C\subset (\mb C(i)\setminus A)\cup (A\Delta C)\subset (\mb A(i)\Delta \mb C(i))\cup (A\Delta C),
		\]
		it follows from~(\ref{trans}) that 
		\[
		\mb A(i)\Delta \mb B(i)\subset (\mb A(i)\Delta\mb C(i))\cup(\mb C(i)\Delta \mb B(i))\subset (\mb A(i)\Delta \mb C(i))\cup (A\Delta C)
		\]
		for every $i\in I$. As $\mb A(i)\Delta \mb B(i)\in D_\nu$, Lemma~\ref{l_cring} implies that $\mb A(i)\Delta \mb B(i)\in\mathcal N$ for every $i\in I$. It follows from Lemma~\ref{ll1}(i) that the set $B = \bigcup_{i\in I}\mb B(i)$ belongs to $D_\mu$ because $\mb B(i)\subset C$ for every $i\in I$. 
		Let $N = A\Delta B$. Then $N\in D_\nu$ and it follows from~(\ref{0}) that $N\subset N'$, where $N' = \bigcup_{i\in I}\mb A(i)\Delta\mb B(i)$. Since $N'\in\sigma(\mathcal N)$, Lemmas~\ref{l_NcapB} and~\ref{l_cring} imply that $N\in\sigma(\mathcal N)$. As $\mathcal N_\nu = \mathcal N$ by Theorem~\ref{p_Ncompl} and $\nu$ is measure-like, we conclude that $N\in\mathcal N$. Thus, $\nu(A) = \mu(B)$ and $\nu(\mb A(i)) = \mu(\mb B(i))$ for every $i\in I$. Since $\mu(B) = \sum_{i\in I}\mu(\mb B(i))$ by the $\sigma$-additivity of $\mu$, we conclude that $\nu(A) = \sum_{i\in I}\nu(\mb A(i))$.  
	\end{proof}

	\begin{corollary}
		Let $\mathfrak A$ be a HTAG. If $\mu$ is an $\mathfrak A$-valued measure, then so are $\mathfrak C\mu$, $\mathfrak S\mu$, and $\overline{\mathfrak C}\mu$.
	\end{corollary}
	\begin{proof}
		The statement holds by Corollary~\ref{c_mucomp}, Proposition~\ref{p_mlike}, and Theorem~\ref{t_meas}. 
	\end{proof}

	Let $\mu$ be an $\mathfrak A$-valued measure. Theorem~\ref{t_meas} implies, in particular, that $\mathfrak C_{\mathcal N}\mu$ is an $\mathfrak A$-valued measure whenever $\mathcal N$ is a completion ring for $\mu$ and a $\sigma$-ring. It should be noted that this is not true for $\sigma$-contents: Example~\ref{e5} shows that the completion of a $\sigma$-content with respect to a $\sigma$-ring may fail to be $\sigma$-additive.  
	
	Let $\mu$ be a measure-like map such that $D_\mu$ is a $\sigma$-ring. Then $\sigma(\mathcal N_\mu) \subset D_\mu$ and it follows immediately from Definition~\ref{d_measurelike} that $\mathcal N_\mu$ is a $\sigma$-ring. 
	
	\begin{proposition}\label{p_sigma}
		Let $\mathfrak A$ be a HTAG and $\mu$ be an $\mathfrak A$-valued measure such that $D_\mu$ is a $\sigma$-ring. 
		\begin{enumerate}
			\item[(i)] Let $\mathcal N$ be a completion ring for $\mu$ and $\nu = \mathfrak C_{\mathcal N}\mu$. Then $\mathcal N$ is a $\sigma$-ring if and only if $\nu$ is an $\mathfrak A$-valued measure and $D_\nu$ is a $\sigma$-ring.
			\item[(ii)] $\mathfrak C\mu$, $\mathfrak S\mu$, and $\overline{\mathfrak C}\mu$ are $\mathfrak A$-valued measures whose domains are $\sigma$-rings.
		\end{enumerate}
	\end{proposition} 
	\begin{proof}
		(i) If $\mathcal N$ is a $\sigma$-ring, then $\nu$ is an $\mathfrak A$-valued measure by Theorem~\ref{t_meas} and $D_\nu$ is a $\sigma$-ring by Corollary~\ref{c_sigma}. If $\nu$ is an $\mathfrak A$-valued measure and $D_\nu$ is a $\sigma$-ring, then $\nu$ is measure-like by Proposition~\ref{p_mlike} and, hence, $\mathcal N_\nu$ is a $\sigma$-ring. It remains to note that $\mathcal N = \mathcal N_\nu$ by Theorem~\ref{p_Ncompl}. 
		\par\medskip\noindent
		(ii) By Proposition~\ref{p_mlike}, $\mu$ is measure-like and, hence, $\mathcal N_\mu$ is a $\sigma$-ring. This implies that $\hat{\mathcal N}_\mu$ is a $\sigma$-ring and, hence, $\mathcal N_0[\mu]$ and $\mathcal N[\mu]$ are $\sigma$-rings by~(\ref{N0mu}) and~(\ref{Nmu}). The statement therefore follows from~(i).
	\end{proof}
	
	If $\mu$ is a positive content (but not a measure) such that $D_\mu$ is a $\sigma$-ring, then the domain of $\mathfrak C\mu$ may fail to be even a $\delta$-ring (see Example~\ref{e3}). 
	
	We conclude this section by giving a sufficient condition ensuring the coincidence of $\overline{\mathfrak C}\mu$ and $\mathfrak C\mathfrak S\mu$.

	\begin{definition}\label{d_decomp}
		A map $\mu$ is called decomposable if $\varnothing\in D_\mu$ and there is a partition $\mb B\colon I\to D_\mu$ of $S[\mu]$ such that
		\begin{enumerate}
			\item[(a)] $D_\mu = \{B\in \hat D_\mu: B\cap\mb B(i)\in D_\mu\mbox{ for every }i\in I\}$,
			\item[(b)] for every $B\in D_\mu$, there is a countable set $J\subset I$ such that the set $B\setminus\bigcup_{i\in J}\mb B(i)$ belongs to $\mathcal N_\mu$. 
		\end{enumerate}
	\end{definition}

	\begin{proposition}
		If $\mu$ is a decomposable measure-like map, then $\overline{\mathfrak C}\mu = \mathfrak C\mathfrak S\mu$.
	\end{proposition}
	\begin{proof}
		In view of Proposition~\ref{p_nat}, it suffices to show that $\mathcal N[\mu] \subset \widehat{\mathcal N_0[\mu]}$ (the opposite inclusion being obvious). Let $N\in \mathcal N[\mu]$ and $\mb B\colon I\to D_\mu$ be as in Definition~\ref{d_decomp}. Since $N\cap\mb B(i)\in \hat{\mathcal N}_\mu$, there is $\mb N\colon I\to \mathcal N_\mu$ such that 
		\[
		N\cap\mb B(i)\subset \mb N(i) \subset \mb B(i) 
		\]
		for every $i\in I$. We claim that the set $N' = \bigcup_{i\in I}\mb N(i)$ belongs to $\mathcal N_0[\mu]$. Let $B\in D_\mu$ and $B' = N'\cap B$. For every $i\in I$, we have $B'\cap\mb B(i) = B\cap \mb N(i)$ and, hence, $B'\cap\mb B(i)\in \mathcal N_\mu$. By Definition~\ref{d_decomp}(a), we conclude that $B'\in D_\mu$. In view of Definition~\ref{d_decomp}(b), there is a countable $J\subset I$ such that $B'\setminus C\in\mathcal N_\mu$, where $C = \bigcup_{i\in J}(B'\cap\mb B(i))$. Since $C\in \hat D_\mu$ and $C\in\sigma(\mathcal N_\mu)$, we have $C\in\mathcal N_\mu$ by~(\ref{mlike}). As $B' = C\cup (B'\setminus C)$, it follows that $B'\in \mathcal N_\mu$ and, therefore, $N'\in \mathcal N_0[\mu]$. This implies that $N\in \widehat{\mathcal N_0[\mu]}$ because $N\subset N'$. Thus, $\mathcal N[\mu] \subset \widehat{\mathcal N_0[\mu]}$. 
	\end{proof}

	In particular, the equality $\overline{\mathfrak C}\mu = \mathfrak C\mathfrak S\mu$ holds for measure-like maps that are $\sigma$-finite in the sense that $S[\mu]$ is a countable union of elements of $D_\mu$. 

	In the case of positive measures, our definition of decomposability is equivalent to that in~\cite{Fremlin1978,FremlinV2} and is somewhat less restrictive than the notion of the direct sum of measure spaces in~\cite{Segal}. It was shown by Fremlin~\cite[Example~8]{Fremlin1978} that positive measures that are localizable in the sense of~\cite{Segal} are not necessarily decomposable, and it seems plausible that $\overline{\mathfrak C}\mu$ and $\mathfrak C\mathfrak S\mu$ may be different for localizable $\mu$. However, we do not know such examples (the measure $\mu$ with $\overline{\mathfrak C}\mu \neq \mathfrak C\mathfrak S\mu$ constructed in Section~\ref{s7} is not localizable).

\section{Examples}
\label{s7}

Let $A$ be a set of pairs. For every $x$, we define the horizontal $x$-section $A_{[x]}$ of $A$ and the vertical $x$-section $A^{[x]}$ of $A$ by the relations $A_{[x]} = \{y : (y,x)\in A\}$ and $A^{[x]} = \{y : (x,y)\in A\}$. We set $\mathrm{Pr}_1 A = \{x: A^{[x]}\neq\varnothing\}$ and $\mathrm{Pr}_2 A = \{x: A_{[x]}\neq\varnothing\}$. 

Given a set $I$, we denote its cardinality by $|I|$.

\begin{lemma}\label{l_full}
	Let $I$ and $J$ be uncountable sets with $|I|<|J|$ and $A\subset I\times J$ be such that $A_{[j]}$ is countable for every $j\in J$. Then $J\setminus A^{[i]}$ is uncountable for some $i\in I$.
\end{lemma}
\begin{proof}
	Suppose $J\setminus A^{[i]}$ is countable for every $i\in I$. Let $K = \bigcap_{i\in I}A^{[i]}$. Since $|I|<|J|$, the equality $K = J\setminus\bigcup_{i\in I}(J\setminus A^{[i]})$ implies that $K\neq\varnothing$. This contradicts the hypothesis because the set $A_{[j]}$ is equal to $I$ and, hence, is uncountable for every $j\in K$.
\end{proof}

Let $I$ be an uncountable set. We define 
\[
\mathcal R_I = \{A\subset I: \mbox{either $A$ or $I\setminus A$ is countable}\}
\]
and let $\fm{\eta_I}{\mathcal R_I}$ be such that, for every $A\in\mathcal R_I$, we have $\eta_I(A) = 0$ if $A$ is countable and $\eta_I(A) = 1$ if $I\setminus A$ is countable. Clearly, $\eta_I$ is a positive measure.

\begin{example}\label{e1}
We construct a positive measure $\mu$ such that $\overline{\mathfrak C}\mu\neq \mathfrak C\mathfrak S\mu$. 

Let $I$ and $J$ be uncountable sets such that $|I|<|J|$. Let the map $\lambda$ be such that
\begin{align}
& D_\lambda = \{B\subset I\times J: \mathrm{Pr}_2 B\mbox{ is finite and }B_{[j]}\in\mathcal R_I\mbox{ for every }j\in J\},\nonumber\\
& \lambda(B) = \sum_{j\in \mathrm{Pr}_2 B} \eta_I(B_{[j]}),\quad B\in D_\lambda.\nonumber
\end{align}
Then $\lambda$ is a positive measure. The set  
\[
\mathcal N = \{N\subset I\times J: \mathrm{Pr}_1 N\mbox{ is countable and }N^{[i]}\in\mathcal R_J\mbox{ for every }i\in I\}
\] 
is obviously a $\sigma$-ring. Since $\mathcal N_\lambda$ consists of all $N\subset I\times J$ such that $\mathrm{Pr}_1 N$ is countable and $\mathrm{Pr}_2 N$ is finite, we have $\mathcal N\cap D_\lambda = \mathcal N_\lambda$ and $N\cap B\in \mathcal N_\lambda$ for every $N\in\mathcal N$ and $B\in D_\lambda$. Hence, $\mathcal N$ is a completion ring for $\lambda$ and $\mathcal N\subset \mathcal N_0[\lambda]$. Theorem~\ref{t_meas} implies that $\mu = \mathfrak C_{\mathcal N}\lambda$ is a positive measure. We claim that $\mathcal N[\mu] \neq \widehat{\mathcal N_0[\mu]}$ and, therefore, $\overline{\mathfrak C}\mu \neq \mathfrak C\mathfrak S\mu$ by Proposition~\ref{p_nat}. As $\lambda$ is complete, we have $\mathcal N[\lambda] = \mathcal N_0[\lambda]$ and Lemma~\ref{l_compl1} implies that
\begin{align}
& \mathcal N[\mu] = \mathcal N[\lambda] = \{N\subset I\times J: N_{[j]}\mbox{ is countable for every }j\in J\},\nonumber\\
& \mathcal N_0[\mu] = \{N\subset I\times J: N^{[i]}\in \mathcal R_J \mbox{ and } N_{[j]}\mbox{ is countable for every }i\in I\mbox{ and }j\in J\}.\nonumber
\end{align}

Let $\fm{\mb N}{I}$ be a disjoint family such that $\mb N(i)\subset J$ and $|\mb N(i)| = |I|$ for every $i\in I$ (such a family exists because $|I|<|J|$ and $|I\times I| = |I|$). Let $N = \bigcup_{i\in I} \{i\}\times \mb N(i)$. Suppose $N\subset\tilde N$, where $\tilde N\in\mathcal N_0[\mu]$. For every $i\in I$, we have $N^{[i]} = \mb N(i)$ and, hence, $N^{[i]}$ is uncountable. This implies that $J\setminus\tilde N^{[i]}$ is countable for every $i\in I$, in contradiction to Lemma~\ref{l_full}. Thus, $N\notin \widehat{\mathcal N_0[\mu]}$. On the other hand, $N\in \mathcal N[\mu]$ because $N_{[j]}$ has at most one element for every $j\in J$. Hence, $\mathcal N[\mu] \neq \widehat{\mathcal N_0[\mu]}$. 

The measure $\mu$ constructed above is similar to that given by Sz\H{u}cs~\cite{Szucs} as an example of a measure that is not localizable in the sense of~\cite{Segal} but has a localizable completion. In the same way as in~\cite{Szucs}, it can be shown that $\mu$ is not localizable.     
\end{example}

\begin{example}\label{e2}
Halmos~\cite[p.~131]{Halmos} gave an example demonstrating that the Radon--Nikodym theorem may break down for non-$\sigma$-finite positive measures. We show that the measures used in~\cite{Halmos} can be represented as suitable $\mathcal N$-completions and rederive the argument for the failure of the Radon--Nikodym theorem in this setting. 

	Let $I$, $J$, $\mathcal N$, $\lambda$, and $\mu$ be as in Example~\ref{e1} and let  
	\[
	\mathcal N' = \{N\subset I\times J: \mathrm{Pr}_2 N\mbox{ is countable and }N_{[j]}\in\mathcal R_I\mbox{ for every }j\in J\}.
	\]
	We define the positive measure $\lambda'$ by setting
	\begin{align}
	& D_{\lambda'} = \{B\subset I\times J: \mathrm{Pr}_1 B\mbox{ is finite and }B^{[i]}\in\mathcal R_J\mbox{ for every }i\in I\},\nonumber\\
	& \lambda'(B) = \sum_{i\in \mathrm{Pr}_1 B} \eta_J(B^{[i]}),\quad B\in D_{\lambda'}.\nonumber
	\end{align} 
	Proceeding as in Example~\ref{e1}, we make sure that $\mathcal N'$ is a completion ring for $\lambda'$.
	As $\mathcal N'$ is a $\sigma$-ring, Theorem~\ref{t_meas} implies that $\mu' = \mathfrak C_{\mathcal N'}\lambda'$ is a positive measure.
	
	Let $\mathcal Q = \Delta(\mathcal N,\mathcal N')$. It is easy to see that $\mathcal Q$ is a $\sigma$-ring and, hence, $\sigma(D_\mu) = \sigma(D_{\mu'}) = \mathcal Q$. As $(N\Delta N')^{[i]} = N^{[i]}\Delta N^{\prime[i]}$ and $(N\Delta N')_{[j]} = N_{[j]}\Delta N'_{[j]}$, we have
	\begin{equation}\label{R}
	\mbox{$B^{[i]}\in \mathcal R_J$ and $B_{[j]}\in\mathcal R_I$ for every $B\in\mathcal Q$, $i\in I$, and $j\in J$.}
	\end{equation}
	
	Let $\mathcal K = D_\mu\cap D_{\mu'}$ and $\fm{\nu}{\mathcal K}$ be such that $\nu(B) = \mu(B) + \mu'(B)$ for every $B\in\mathcal K$. Then $\mathcal K$ is a $\delta$-ring and $\nu$ is a positive measure. 
	Since $D_\lambda\subset \mathcal N'$ and $D_{\lambda'}\subset \mathcal N$, we have $D_\lambda\cup D_{\lambda'}\subset \mathcal K \subset \mathcal Q$ and, hence, $\sigma(\mathcal K) = \mathcal Q$. Note that $\nu$ is complete because $\mathcal N_\nu = \mathcal N\cap\mathcal N'$ is the set of all countable subsets of $I\times J$.
	
	Clearly, $\mu$ is absolutely continuous with respect to $\nu$ in the sense that $\mathcal N_\nu\subset\mathcal N_\mu$ (and, therefore, $\mathcal N[\nu]\subset\mathcal N[\mu]$). We shall show that, nevertheless, $\mu$ has no density with respect to $\nu$:
	\begin{enumerate}
		\item[(S)] There is no function $f\colon I\times J\to[0,\infty)$ such that $f$ is $\nu$-integrable on $B$ and $\mu(B) = \int_B f(s)\,d\nu(s)$ for every $B\in\mathcal K$.
	\end{enumerate}   
	
	Suppose such an $f$ exists. Let $A = \{s\in I\times J: f(s) = 0\}$. If $B\in\mathcal K$, then $B\setminus A\in\mathcal K$ because $f$ is $\nu$-measurable on $B$. Clearly, $\int_B f(s)\,d\nu(s) = 0$ if and only if $\nu(B\setminus A) = 0$. Hence,
	\begin{equation}\label{eq}
	\mu(B) = 0 \Leftrightarrow \nu(B\setminus A) = 0,\quad B\in\mathcal K.
	\end{equation}	

	Let $i\in I$, $B = \{i\}\times J$, and $C = B\setminus A$. Since $C\in \mathcal K$, we have $C^{[i]}\in\mathcal R_J$ by~(\ref{R}). This implies that $C\in D_{\lambda'}$ and, hence, $C\in \mathcal N$. It follows that $\mu(C) = 0$ and
	\[
	\nu(C) = \mu'(C) = \lambda'(C) = \eta_J(C^{[i]}).
	\]
	By~(\ref{eq}), we conclude that $\eta_J(C^{[i]}) = 0$ because $B\in\mathcal N$ and $\mu(B) = 0$. This means that $J\setminus A^{[i]}$ is countable.
	
	Now let $j\in J$ and $B = I\times\{j\}$. Again, we set $C = B\setminus A$. Since $C\in \mathcal K$, we have $C_{[j]}\in\mathcal R_I$ by~(\ref{R}). This implies that $C\in D_\lambda$ and, hence, $C\in \mathcal N'$. It follows that $\mu'(C) = 0$ and
	\[
	\nu(C) = \mu(C) = \lambda(C) = \eta_I(C_{[j]}).
	\]
	By~(\ref{eq}), we conclude that $\eta_I(C_{[j]}) \neq 0$ because $\mu(B) = \lambda(B) = 1$. As $A_{[j]} = I\setminus C_{[j]}$, this means that $A_{[j]}$ is countable.
	
	Thus, $J\setminus A^{[i]}$ is countable for every $i\in I$ and $A_{[j]}$ is countable for every $j\in J$, in contradiction to Lemma~\ref{l_full}. This completes the proof of~(S).
	
	We note that $\mu'$ also has no density with respect to $\nu$ despite being absolutely continuous (if $f$ were such a density, then $1-f$ would be a density for $\mu$).
	
	The treatment in~\cite{Halmos} involves extended real-valued measures defined on the $\sigma$-ring $\mathcal Q$. To pass to our setting, one has to restrict these measures to $\delta$-rings, where they are finite.
\end{example}	

\begin{example}\label{e3}
We construct a positive content $\mu$ such that $D_\mu$ is a $\sigma$-algebra and $D_{\mathfrak C\mu}$ is not a $\delta$-ring. 

Let $\mathcal U$ be a non-principal ultrafilter on an infinite countable set $I$. Let $\mathcal P$ be the set of all subsets of $I$ and $\fm{\lambda}{\mathcal P}$ be such that, for every $C\in\mathcal P$, we have $\lambda(C) = 1$ if $C\in\mathcal U$ and $\lambda(C) = 0$ if $C\notin \mathcal U$. Then $\lambda$ is a positive content. Since $\lambda$ is complete, it cannot serve as the required example. We therefore construct a modification $\mu$ of $\lambda$ as follows. Let $J = \{0,1\}$ and $\mathcal Q$ be the set of all sets of the form $C\times J$, where $C\in \mathcal P$. Clearly, $\mathcal Q$ is a $\sigma$-algebra. Let $\fm{\mu}{\mathcal Q}$ be such that $\mu(C\times J) = \lambda(C)$ for every $C\in\mathcal P$. Then $\mu$ is a positive content and $\hat{\mathcal N}_\mu = \{N\subset I\times J: \lambda(\mathrm{Pr}_1N)=0\}$. Let $\nu = \mathfrak C\mu$. Since $\mathcal U$ is non-principal, we have $\lambda(\{i\}) = 0$ for every $i\in I$. This implies that $\{(i,0)\}\in \hat{\mathcal N}_\mu$ and, hence, $\{(i,0)\}\in D_\nu$ for every $i\in I$. Let $A = I\times \{0\} = \bigcup_{i\in I}\{(i,0)\}$. For every $B\in D_\mu$, we have $\mathrm{Pr}_1 (A\Delta B) = I$ and, therefore, $A\notin D_\nu$. Thus, $D_\nu$ is not closed under countable unions. Since $D_\nu$ is an algebra by Theorem~\ref{p_Ncompl}, it follows that $D_\nu$ is not closed under countable intersections.  
\end{example}

\begin{example}\label{e4}
We construct $\mu$ and $\mathcal N$ with the following properties: $\mu$ is a positive measure, $D_\mu$ is a $\sigma$-algebra, and $\mathcal N$ is a completion ring for $\mu$ and a $\delta$-ring; at the same time, the map $\nu = \mathfrak C_{\mathcal N}\mu$ is not $\sigma$-additive and $D_\nu$ is not a $\delta$-ring.

Let $I$ be an infinite set and the map $\mu$ be such that $D_\mu = \{\{\varnothing\},\{I\}\}$, $\mu(\varnothing) = 0$, and $\mu(I) = 1$. Then $\mu$ is a positive measure and $D_\mu$ is a $\sigma$-algebra. Let $\mathcal N$ be the set of all finite subsets of $I$. Then $\mathcal N$ is obviously a $\delta$-ring and a completion ring for $\mu$. Let $\nu = \mathfrak C_{\mathcal N}\mu$. Then 
$D_\nu = \{A: A\in\mathcal N\mbox{ or }I\setminus A\in\mathcal N\}$, and we have $\nu(A) = 0$ if $A\in\mathcal N$ and $\nu(A)=1$ if $I\setminus A\in\mathcal N$. It is clear that $D_\nu$ is not a $\delta$-ring and $\nu$ is $\sigma$-additive if and only if $I$ is uncountable.
\end{example}

\begin{example}\label{e5}
We construct a positive $\sigma$-content $\mu$ and a completion ring $\mathcal N$ for $\mu$ such that $\mathcal N$ is a $\sigma$-ring and $\mathfrak C_{\mathcal N}\mu$ is not $\sigma$-additive. 

Let $\mathcal Q$ be the set of all finite unions of semi-intervals on $\R$ of the form $(a,b]$, where $a<b$, and let $\mu$ be the restriction of the Lebesgue measure on $\R$ to $\mathcal Q$. Since $\mathcal Q$ is a ring, $\mu$ is a positive $\sigma$-content. Let $\mathcal N$ be the set of all subsets of $\R\setminus\mathbb Q$, where $\mathbb Q$ is the set of rational numbers. Clearly, $\mathcal N$ is a $\sigma$-ring and a completion ring for $\mu$. Let $\nu = \mathfrak C_{\mathcal N}\mu$ and $A = (0,1]\cap\mathbb Q$. Then $A\in D_\nu$ and $\nu(A) = 1$. Since $A$ is countable, there is a bijection $\tau$ of $\N$ onto $A$. Let $\mb A\colon \N\to\mathcal Q$ be such that $\mb A(n) = (\tau(n)-2^{-n-1},\tau(n)]$ for every $n\in\N$. Then $\nu(\mb A(n)) = 2^{-n-1}$ for every $n\in \N$ and, hence, $\sum_{n\in \N}\nu(\mb A(n)) = 1/2 < \nu(A)$. As $A\subset \bigcup_{n\in \N}\mb A(n)$, it follows that $\nu$ is not $\sigma$-subadditive and, therefore, is not $\sigma$-additive.
\end{example}

\section*{Acknowledgements} 
M.S.~Smirnov was supported by the Moscow Center of Fundamental and Applied Mathematics at INM RAS (Agreement with the Ministry of Education and Science of the Russian Federation No.~075-15-2022-286).

\end{document}